\documentclass[a4paper, 10 pt, reqno]{amsproc}

\usepackage{defs}

\begin{document}
\title[Distributed alternating gradient descent over a network]{Distributed alternating gradient descent for convex semi-infinite programs over a network}

\author[A. Aravind]{Ashwin Aravind\orcidlink{0000-0002-6412-5772}} 
\author[D. Chatterjee]{Debasish Chatterjee\orcidlink{0000-0002-1718-653X}}
\author[A. Cherukuri]{Ashish Cherukuri\orcidlink{0000-0002-7609-5080}} 
\thanks{This work was partially supported by Ministry of Human Resource Development, Govt. of India. D. Chatterjee acknowledges partial support of the SERB MATRICS grant MTR/2022/000656. A. Aravind and D. Chatterjee are with the Centre for Systems and Control Engineering, Indian Institute of Technology Bombay, India. (e-mails: (a.aravind, dchatter)@iitb.ac.in). A. Cherukuri is with the Engineering and Technology Institute Groningen, University of Groningen, Groningen, The Netherlands (e-mail: a.k.cherukuri@rug.nl).}

\begin{abstract}
This paper presents a first-order distributed algorithm for solving a convex semi-infinite program (SIP) over a time-varying network. In this setting, the objective function associated with the optimization problem is a summation of a set of functions, each held by one node in a network. The semi-infinite constraint, on the other hand, is known to all agents. The nodes collectively aim to solve the problem using local data about the objective and limited communication capabilities depending on the network topology. Our algorithm is built on three key ingredients: consensus step, gradient descent in the local objective, and local gradient descent iterations in the constraint at a node when the estimate violates the semi-infinite constraint. The algorithm is constructed, and its parameters are prescribed in such a way that the iterates held by each agent provably converge to an optimizer. That is, as the algorithm progresses, the estimates achieve consensus, and the constraint violation and the error in the optimal value are bounded above by vanishing terms. Simulation examples illustrate our results.
\end{abstract}

\maketitle
\allowdisplaybreaks
\section{Introduction}
\label{sec:introduction}

A semi-infinite program (SIP) is a mathematical optimization problem featuring an objective function and a feasible set described by infinitely many inequalities. The presence of an infinite number of constraints makes the task of numerically solving SIPs challenging. These problems are encountered in a variety of fields, see~\cite{MC-SG:18} for a detailed discussion. Applications include, for example, robotics and control~\cite{EP-DQM-DMS:84}, finance~\cite{AG-MD-CZ:24}, and other areas~\cite{ABT-LEG-ASN:09,DB-DB-CC:10}. These problems are particularly challenging to solve due to the presence of possibly uncountable number of constraints. Several numerical methods are proposed to solve SIPs exactly or approximately, see surveys~\cite{RH-KOK:93,MAL-GS:2007, MAG-MAL:18} for detailed expositions. A majority of the work on numerical solutions to SIPs focuses on centralized approaches. However, advancements in communication technologies as well as increased emphasis on large-scale network-centric systems have facilitated the need for distributed algorithms for solving optimization problems~\cite{GN-IN-AC:19}. Naturally, applications exist for distributed frameworks where there are an infinite number of constraints. For example, robust distributed control problems~\cite{YK-JPH:11, YW-CM:22} can be reformulated as distributed SIPs. Similarly, the authors in~\cite{AC-AZ-GB-ARH:22} study the distributionally robust optimization over a network as a distributed SIP. In such settings, the difficulties associated with distributed computing and the infinite number of constraints compound each other, thus demanding specialized algorithms. Our objective in this article is to design a distributed algorithm for solving semi-infinite programs (SIPs) with convex data, where the objective function is a sum of local functions, each privately known to its corresponding agent in a network, while the semi-infinite constraint is shared globally among all agents. 

\subsection{Related works} Distributed SIPs have appeared in few works and existing methods can be mainly categorized as  cutting-plane methods, and scenario-based methods. In this article, we introduce a third category based on first-order methods. Below we briefly review representative works in each category and highlight the distinctions relative to our approach.

Cutting-plane methods construct outer approximations of the feasible set of the problem; this set may be defined by infinitely many constraints. A distributed cutting-plane approach was developed in~\cite{MB-GN-FA:12} for optimization problems with a common linear objective and node-specific semi-infinite constraints. A combination of cutting-plane ideas and distributed ADMM was proposed in~\cite{AC-AZ-GB-ARH:22} for a related problem structure. In both cases, convergence guarantees are asymptotic with no known explicit convergence rates. Moreover, the number of constraints maintained at each agent can grow with iterations, leading to significant computational and memory burdens. In contrast, our approach offers explicit convergence rate guarantees for both objective suboptimality and feasibility violation while requiring each agent to store only local data and a current estimate. A cutting-plane-based scheme with convergence rate guarantees comparable to our method was presented in~\cite{JF-XW:23}, along with finite-time termination using zeroth-order stopping conditions. Unlike~\cite{JF-XW:23}, which requires potentially expensive projection operations onto constraint sets, our method circumvents this challenge by relying on first-order information of the constraint function to maintain feasibility.

Scenario-based~\cite{MC-SG:18} methods, such as those explored in~\cite{LC-VS-FB-GCC:14, KM-AF-SG-MP:18, KY-RT-PX:19}, rely on random sampling from the constraint index set to create a finitely described outer approximation of the feasible set. Convergence guarantees for these methods are probabilistic and typically asymptotic, and the absence of convergence rates at the same level of granularity limits their applicability. Furthermore, achieving feasibility with high-probability requires a large number of samples especially in high dimensions, resulting in substantial memory and computation costs. Finally,~\cite{XW-JF:23} developed a finite-time convergence scheme for distributed SIPs based on a distributed upper and lower bounding procedure. While this method achieves finite-time convergence guarantees, but it requires solving two finitely constrained distributed optimization problems at each iteration, which is challenging.

In summary, compared to existing approaches, our method provides explicit convergence guarantees, maintains low memory and computational overhead, and utilizes only first-order information without requiring projection based on the constraint function. We note that random projection methods, introduced in~\cite{SL-AN:13,SL-AN:16}, could be adapted to solve SIPs. However, such an adaptation would require computing expensive projections at each iteration and would only provide asymptotic convergence, without any guarantees on the rate of convergence for either optimality or feasibility.

\subsection{Contributions} Our approach combines the standard first-order distributed optimization framework~\cite{AN-AO:09} with the alternating gradient descent schemes for constrained optimization problems~\cite{AB-ABT-NGB-LT:10, BW-WBH-SZ:20}, where the latter set of algorithms employ gradient descent either for the objective or the constraint function depending on the feasibility or lack of it, respectively, of the iterate. Our algorithm is built on three blocks: (a)  a consensus step to bring iterates close to each other; (b) a gradient step in the objective; and (c) several gradient steps in the constraint when faced with infeasibility of the iterates. Our algorithm overcomes the challenge of projection onto a set defined by  an infinite number of constraints in this way. We show that the iterates held by each node in the algorithm get close to each other as iterations progress and furthermore, two properties hold. First, the feasibility violation of the estimate at the $K$-th iteration is of the order $\bigo{\frac{1}{\sqrt{\threshold}}}$, and second, the error in the objective value is also of the order $\bigo{\frac{1}{\sqrt{\threshold}}}$, and thus, decays asymptotically to zero as $\threshold\to\infty$. A numerical experiment is provided to illustrate the performance of this framework and to compare it with two standard methods from the literature, namely, the distributed cutting-surface ADMM~\cite{AC-AZ-GB-ARH:22} and the distributed scenario approach~\cite{KM-AF-SG-MP:18}.

\subsection{Organization} Section~\ref{sec:problem_statement} introduces the semi-infinite program over a network. In Section~\ref{sec:algorithm}, we provide the distributed algorithm along with its detailed description. The convergence analysis is presented in Section~\ref{sec:convergence} and Section~\ref{sec:simulations} showcases illustrative examples. 

\subsection{Notation} The set $\N{}{}:=\setdef{1,2,\cdots}$ is the set of natural numbers. For a number $n\in\N{}{}$, $\until{n}:=\setdef{1,2,\cdots,n}$. Sets containing real  and non-negative real numbers are represented by $\R{}{}$ and $\RP{}$, respectively. For $n\in\R{}{}$, we denote by $\grtint{n}$ the largest integer that is less than or equal to $n$. The vector space \(\R{n}{}\) is assumed to be equipped with the standard inner product \(\inpr{a}{b} := \sum_{j=1}^n a_j b_j\) for every \(a,b \in \R{n}{}\). For a closed set $\F$ and any point $a\in\F$ we denote the normal cone at $a$ as  $\nfset{\F}\pbrack{a} :=\setdeff{b}{\inpr{b}{a-c}\leq 0 \text{ for all }c\in\F}$. By $\norm{\cdot}$ we denote the Euclidean norm on $\R{n}{}$, the associated dual norm is represented by $\norm{\cdot}_{*}$, which for the Euclidean norm is itself. For a closed set $\dumset$, the Euclidean projection to this set from an external point $\dumb{}{}$ is, $\proj{\dumset}{\dumb{}{}}:=\argmin_{\duma{}{}\in\dumset}\norm{\duma{}{}-\dumb{}{}}$. For a mapping $\confunc{}:\dvset{}\times \uvset{} \to\R{}{}$, the set $\lset{\confunc{}}{\alpha}:=\setdeff{\dv{}{}\in\dvset{}}{\max_{\uv{}{}{}\in\uvset{}}\confunc{}{\dv{}{}, \uv{}{}}=\alpha}$ denotes its marginal $\alpha$-level set, and $\gradg{}$ denotes its derivative with respect to the first variable whenever defined. For a subdifferentiable map $F_{\inode}:\dvset{}\to\R{}{}$, a vector $\gradf{\inode}{a}$ is its subgradient at the point $a\in\dvset{}$ if $\gradf{i}{a}^{\top}(b-a)\leq F_{i}(b)-F_{i}(a)$ for all $b\in\dvset{}$. The subdifferential of $F_i$ at $a$, denoted by $\sgradf{\inode}{a}$, is the set of all subgradients at $a$, i.e., $\sgradf{\inode}{a}:=\setdeff{c}{c^{\top}(b-a)\leq F_{i}(b)-F_{i}(a) \text{ for all } b\in\dvset{}}$.

\section{Problem statement}
\label{sec:problem_statement}
Consider a network of $\tnode\in \N{}{}$ computational nodes connected via a time-dependent directed graph $\graph{\oc} := (\node, \edge{\oc})$, where $\oc\in \N{}{}$ denotes the time. The vertex set $\node$ contains identifiers $\inode\in\until{\tnode}$ assigned to each node, and the edge set $\edge{\oc} \subset \node\times \node$ is the collection of directed edges between these nodes at a given time $\oc$. The weights associated with edges in $\edge{\oc}$ govern the communication over the network at time $\oc$. For a pair of nodes $(i,j)\in\node\times\node$  the weight associated with them at a time $\oc$ is $\wtgraph{\oc}{\inode}{\jnode}\in\lcrc{0}{1}$ and the adjacency matrix $\wtmat{\oc}\in\R{\tnode\times\tnode}{}$ contains weights for all such pairs. We have  $\wtgraph{\oc}{\inode}{\jnode}>0$ if and only if   $(i,j)\in\edge{\oc}$. At any given time, the neighboring nodes for any node are those nodes that send information to it. We denote the set containing neighboring nodes of an arbitrary node $\inode\in\node$ at time $\oc$ by $\neigh{\oc}{\inode}:=\setdeff{\jnode\in\node}{\wtgraph{\oc}{\inode}{\jnode}>0}$. The network transition matrix is given by  $\graphflow{t}{s}:=\wtmat{t}\wtmat{t-1}\cdots\wtmat{s}$ (where $t,s\in\N{}{}$ denote time instances satisfying $t>s\geq 1$).
\begin{assumption}
\label{asmp:graph}
\longthmtitle{Properties of the graph}
The network graph is stipulated to possess the properties:
\begin{enumerate}
\item[\namedlabel{gd1}{(G1)}] Each matrix $\wtmat{\oc}$ is assumed to be doubly stochastic, that is, $\sum_{\jnode\in\node}\wtgraph{\oc}{\inode}{\jnode}=\sum_{\jnode\in\node}\wtgraph{\oc}{\jnode}{\inode}=1$ for all $\inode\in\node$. 
\item[\namedlabel{gd2}{(G2)}] There exists a scalar $0<\minwt<1$ such that if $\wtgraph{\oc}{\inode}{\jnode}>0$ for nodes $\inode$, $\jnode\in\node$ and time $\oc$,  then  $\wtgraph{\oc}{\inode}{\jnode}\geq\minwt$.
\item[\namedlabel{gd3}{(G3)}] The sequence of graphs $\graph{\oc}$ is $\strong$-strongly connected, meaning that there exists $\strong>0$ such that the graph $\graph{\oc}_{\strong}:=\left(\node,\edge{\oc}_{\strong}\right)$ with $\edge{\oc}_{\strong} := \bigcup_{\tau=0}^{\strong-1}\edge{\oc+\tau}$ is strongly connected for every $\oc\geq 1$.~\footnote{A directed graph is strongly connected if for every pair of nodes in the graph there exists a path connecting these nodes in both directions.}  
\end{enumerate}
\end{assumption}

We consider the following optimization problem over the set of nodes $\node$ given by
\begin{equation}
\label{eq:sip_central}
\begin{aligned}
    &\min_{\dv{}{}} &&\sum_{\inode=1}^{\tnode} \objfunc{\inode}{\dv{}{}} \\
    &\sbjto &&\begin{cases}
    \dv{}{}\in \dvset{}, \\
    \confunc{}{\dv{}{},\uv{}{}} \leq 0 \text{ for all } \uv{}{}\in\uvset{},
    \end{cases}
\end{aligned}
\end{equation}
with the following data:
\begin{enumerate}
\item[\namedlabel{pd1}{(\ref{eq:sip_central}a)}] the \emph{domain} $\dvset{}\subset\R{\dvdim}{}$ has a non-empty interior and is compact and convex and its diameter is $\setdia:=\max_{\duma{}{},\dumb{}{}\in\dvset{}} \norm{\duma{}{}-\dumb{}{}}$;
\item[\namedlabel{pd2}{(\ref{eq:sip_central}b)}] the \emph{objective function} is $\mapdeff{\dv{}{}}{\dvset{}}{\objfunc{}{\dv{}{}}:=\sum_{\inode=1}^{\tnode} \objfunc{\inode}{\dv{}{}}}{\R{}{}}$, where each $\mapdeff{\dv{}{}}{\dvset{}}{\objfunc{\inode}{\dv{}{}}}{\R{}{}}$ is convex and subdifferentiable;
\item[\namedlabel{pd3}{(\ref{eq:sip_central}c)}] the \emph{constraint index} set $\uvset{}\subset\R{\uvdim}{}$  is compact;
\item[\namedlabel{pd4}{(\ref{eq:sip_central}d)}] the \emph{constraint function} $\mapdeff{\pbrack{\dv{}{},\uv{}{}}}{\dvset{}\times\uvset{}}{\confunc{}{\dv{}{},\uv{}{}}}{\R{}{}}$ is continuous, and is convex and continuously differentiable in $\dv{}{}$ for each fixed $\uv{}{}$.
\end{enumerate}

The optimization problem~\eqref{eq:sip_central} is a convex semi-infinite program (SIP), and due to~\ref{pd2}, the feasible set $\F:=\setdeff{\dv{}{}\in\dvset{}}{\confunc{}{\dv{}{},\uv{}{}} \leq 0 \text{ for all } \uv{}{}\in\uvset{}}$ is convex and depending on the constraint index set $\uvset{}$, the family of inequality constraints $\setdeff{\confunc{}{\dv{}{},\uv{}{}}\leq 0}{\uv{}{}\in\uvset{}}$ may contain infinitely many constraints. In the setup considered in this article, the nodes in the network aim to solve~\eqref{eq:sip_central} cooperatively in a situation where each node has only partial access to the objective function. Specifically, a node $\inode\in\node$ has access only to a \emph{local} objective function $\objfunc{i}$, the domain $\dvset{}$ and the constraint function $\confunc{}$. The following additional assumptions are made regarding problem~\eqref{eq:sip_central}:
\begin{assumption}
\longthmtitle{Regularity of problem~\eqref{eq:sip_central} data}
\label{asmp:data}
We assume the following:
\begin{enumerate}
    \item[\namedlabel{ad1}{(A1)}] 
    The \emph{feasible set} for problem~\eqref{eq:sip_central} given by $\F$ contains a point that satisfies the semi-infinite constraint strictly, that is, there exists $\bar{x} \in \F$ with $f(\bar{x},u) < 0$ for all $u \in \mathbb{U}$.
    \item[\namedlabel{ad2}{(A2)}] The subgradients of the local objective functions are bounded in the domain $\dvset{}$, $\dvgradf:=\max_{i\in\node}\max_{\dv{}{}\in\dvset{}} \norm{\gradf{i}{\dv{}{}}}<+\infty$.
    \item[\namedlabel{ad3}{(A3)}] The gradient of the constraint function is bounded on $\dvset{}$, i.e., $\dvgradg:=\max_{\dv{}{}\in\dvset{}, \uv{}{}\in\uvset{}} \norm{\gradg{}{\dv{}{}, \uv{}{}}}<+\infty$. We define $\maxgrad:=\max\{\dvgradf,\dvgradg\}$.
\end{enumerate}
\end{assumption}
 We denote the set of optimizers for the problem~\eqref{eq:sip_central} by $\solset$. Depending on the nature (eg., strong or strict convexity) of the objective function/constraints, $\solset$ may or may not be a singleton. In the sequel, we establish an algorithm with provable guarantees to solve the convex SIP~\eqref{eq:sip_central} over a time-varying network via distributed computations.

\section{Distributed alternating gradient descent}
\label{sec:solution}
Our Algorithm~\ref{algo:dis_com} is a consensus-based first-order primal approach to solve convex SIP~\eqref{eq:sip_central} over a time-varying $\strong$-strongly connected network. It is inspired by the \emph{CoMirror descent} method~\cite{AB-ABT-NGB-LT:10} originally developed to solve inequality-constrained convex optimization problems and later extended to convex SIPs in~\cite{BW-WBH-SZ:20}. 

\subsection{The algorithm}
\label{sec:algorithm}
\begin{algorithm}[!ht]
\SetAlgoLined
\DontPrintSemicolon
\SetKwInOut{ini}{Initialize}
\SetKwInOut{giv}{Data}
\tcc{Computations at a node $\inode\in\node$}
\giv{Threshold number of iterations $\threshold$, feasibility tolerances $\seqdef{\errseq{\oc}:=\frac{1}{\sqrt{\oc}}}{\oc\in \until{\threshold}}$, outer stepsizes $\seqdef{\lrate{\oc}:=\frac{\setdia}{\sqrt{\oc}}}{\oc\in \until{\threshold}}$, inner stepsizes $\seqdef{\loclrate{\inode}{\oc}{\ic}}{\ic\geq 1}$ for each $\oc\in\until{\threshold}$.}
\ini{ Initial guess $\dv{1}{\inode}\in\dvset{}$ for all $\inode\in\node$, $\fset{}:=\setdef{\grtint{\frac{\threshold}{2}},\grtint{\frac{\threshold}{2}}+1,\cdots,\threshold}$.} 
\BlankLine
\For{$\oc=1,2,\cdots,\threshold$}{
    Receive $\dv{\oc}{\jnode}$ from all $\jnode\in\neigh{\oc}{\inode}$ and send $\dv{\oc}{\inode}$ to all $\jnode\in\node$ such that $\inode\in\neigh{\oc}{\jnode}$ \;
    Set $\cdv{\oc}{\inode} \gets \sum_{\jnode=1}^{\tnode} \wtgraph{\oc}{\inode}{\jnode} \dv{\oc}{\jnode}$ \;
    Compute $\zdv{\oc}{\inode} \gets \proj{\dvset{}}{\cdv{\oc}{\inode}-\lrate{\oc}\gradf{\inode}{\cdv{\oc}{\inode}}}$ \; 
    Compute $\uv{\oc}{\inode} \gets \argmax_{\uv{}{}\in\uvset{}} \confunc{}{\zdv{\oc}{\inode},\uv{}{}}$ \;
    Set $\ic \gets 1$, $\bzdv{\inode}{\oc}{\ic} \gets \zdv{\oc}{\inode}$ and $\buv{\inode}{\oc}{\ic}\gets\uv{\oc}{\inode} $\;    
    Set $\ball{\oc}{\inode} \gets \setdeff{\dv{}{}\in\dvset{}}{\norm{\dv{}{}-\zdv{\oc}{\inode}}\leq \lrate{\oc}\dvgradf+\frac{\errseq{\oc}}{\mingrad}}$ \;
    \While{$\confunc{}{\bzdv{\inode}{\oc}{\ic},\buv{\inode}{\oc}{\ic}} > \errseq{\oc+1}$}{ 
    Set $\loclrate{\inode}{\oc}{\ic} \gets \frac{\confunc{}{\bzdv{\inode}{\oc}{\ic},\buv{\inode}{\oc}{\ic}}}{\norm{\gradg{}{\bzdv{\inode}{\oc}{\ic}, \buv{\inode}{\oc}{\ic}}}^2}$ \;
    Compute $\bzdv{\inode}{\oc}{\ic+1} \gets \proj{\ball{\oc}{\inode}}{\bzdv{\inode}{\oc}{\ic}-\loclrate{\inode}{\oc}{\ic}\gradg{}{\bzdv{\inode}{\oc}{\ic}, \buv{\inode}{\oc}{\ic}}} $ \;
    Compute $\buv{\inode}{\oc}{\ic+1} \gets \argmax_{\uv{}{}\in\uvset{}} \confunc{}{\bzdv{\inode}{\oc}{\ic+1},\uv{}{}}$ \;
    Set $\ic \gets \ic+1$\;
    }
    Set $\dv{\oc+1}{\inode}\gets\bzdv{\inode}{\oc}{\ic}$\;
}
\Return $\;\bdv{\threshold}{\inode}\gets\frac{1}{\sum_{\oc\in\fset{}}\lrate{\oc}}\sum_{\oc\in\fset{}}\lrate{\oc}\dv{\oc+1}{\inode}$
\caption{Distributed Alternating Gradient Descent}
\label{algo:dis_com}
\end{algorithm}

\subsubsection*{An informal description of Algorithm~\ref{algo:dis_com}}

Each node holds and updates an estimate of an optimizer of~\eqref{eq:sip_central}, this process is elaborated in detail below:
\begin{itemize}
    \item The algorithm is written from the perspective of an arbitrary node $\inode \in \node$, but requires all nodes to agree a priori on the number of iterations $\threshold$, the feasibility tolerances $\seqdef{\errseq{\oc}}{\oc\geq 1}$, and the stepsizes $\seqdef{\lrate{\oc}}{\oc\geq 1}$.
    \item The estimate of the optimizer at node $\inode$ is initialized as $\dv{1}{\inode} \in \dvset{}$ and the nodes exchange their estimates with their neighbors in Step 2. 
    \item In Step 3, a cumulative estimate $\cdv{\oc}{\inode}$ is obtained by averaging the neighbors' estimates using the edge weights $\pbrack{\text{present in }\wtmat{\oc}}$ of the network as weighting factors.  
    \item Projected subgradient descent is performed at the cumulative estimate in Step 4, we term this as the outer descent.
    \item Since the outer descent step is not guaranteed to result in a feasible estimate, Steps  6 to 13 are designed to ensure feasibility through a sequence of gradient steps on the constraint function. 
    \item In Step 6, the inner loop counter $\ic$ is initialized as 1, and the values of the estimate $\zdv{\oc}{\inode}$ and $\uv{\oc}{\inode}$ are copied to $\bzdv{\inode}{\oc}{1}$ and $\buv{\inode}{\oc}{1}$, respectively, for local computation. 
    \item A closed Euclidean ball $\ball{\oc}{\inode}$ is defined around the estimate $\zdv{\oc}{\inode}$ in Step 7, and is used as the feasible set for the projected gradient descent in the inner loop. 
    \item Step 8 checks the feasibility of the current estimate by verifying whether $\confunc{}{\bzdv{\inode}{\oc}{\ic},\buv{\inode}{\oc}{\ic}}>\errseq{\oc}$. If the estimate is feasible, the computations at the node loop to Step 2; otherwise, local computations based on  gradient descent are performed on the  constraint function, using the gradient $\gradg{}{\bzdv{\inode}{\oc}{\ic}, \buv{\inode}{\oc}{\ic}}$. 
    \item The stepsize for these computations follows the Polyak stepsize rule~\cite[Section 8.2.2]{AB:17}: $\loclrate{\inode}{\oc}{\ic} \gets \frac{\confunc{}{\bzdv{\inode}{\oc}{\ic},\buv{\inode}{\oc}{\ic}}}{\norm{\gradg{}{\bzdv{\inode}{\oc}{\ic}, \buv{\inode}{\oc}{\ic}}}^2}$. A new estimate is then obtained in Step 10. We refer to Steps 9 and 10 as the inner descent.
    \item The inner descent aims to obtain a feasible estimate at the node. For every outer iteration $\oc$, the loop comprising Steps 8 through 13 is executed until a feasible estimate is obtained. 
    \item Once the node's estimate becomes feasible, the value of the local variable $\bzdv{\inode}{\oc}{\ic}$ is transferred to $\dv{\oc+1}{\inode}$ as the next estimate of the optimizer at the node, after which the computation returns to Step 2.
    \item This sequence of computations repeats until a threshold number $\threshold$ of outer iterations is reached. The algorithm then returns the weighted average of the estimates ($\bdv{}{\inode}$) as the final estimate of the optimizer at node $\inode$.
\end{itemize}

\begin{remark}\longthmtitle{Finite number of constraints}
\label{rem:finiteconstraints}
Note that Algorithm~\ref{algo:dis_com} can be easily modified to handle constraints of the form $\confunc{}{\dv{}{}}\leq 0$. In this case, Step 5 and Step 11 become redundant since the feasibility of the estimate can be directly checked at Step 8 by evaluating $\confunc{}{\bzdv{\inode}{\oc}{\ic}}$. Similarly, in Step 9, the stepsize of the inner descent is updated to $\loclrate{\inode}{\oc}{\ic} \gets \frac{\confunc{}{\bzdv{\inode}{\oc}{\ic}}}{\norm{\gradgg{}{\bzdv{\inode}{\oc}{\ic}}}^2}$, and the projected gradient descent at Step 10 uses $\gradgg{}{\bzdv{\inode}{\oc}{\ic}}$ as the gradient. Despite these modifications, the analysis of the algorithm remains unchanged, and the guarantees established in the subsequent section continue to hold.
\end{remark}

\begin{remark}
\longthmtitle{Projection onto the closed ball} 
\label{rem:projection}
For an arbitrary node $\inode\in\node$, Step 10 of Algorithm~\ref{algo:dis_com} consists of projected gradient descent iteration with projection onto the closed Euclidean ball $\ball{\oc}{\inode}$. From an external point $z$, this projection may be easily calculated as $\proj{\ball{\oc}{\inode}}{z}=\zdv{\oc}{\inode}+\pbrack{\lrate{\oc}\dvgradf+\frac{\errseq{\oc}}{\mingrad}}\frac{z-\zdv{\oc}{\inode}}{\norm{z-\zdv{\oc}{\inode}}}$. 
\end{remark}
\subsection{Convergence analysis of Algorithm~\ref{algo:dis_com}}
\label{sec:convergence}
In this section, we present the convergence guarantees of the main algorithm. We derive an upper bound on the suboptimality gap, measured by the difference between  the objective value corresponding to the returned estimate at each node and the global optimal value. We also establish guarantees on the feasibility of these estimates. To build toward the main result stated in Theorem~\ref{thm:optimality}, we first present several supporting results. Additionally, in Proposition~\ref{prop:maxstep}, we show that the number of inner descent steps required at any node to obtain a feasible estimate  (determined by the predefined feasibility tolerance sequence) is finite. 

Due to their relevance in subsequent analysis, we introduce the following notations: $\maxstep{\inode}{\oc}$ represents the number of inner iterations required by the node $\inode\in\node$ at the $\oc^{\text{th}}$ iteration of the outer loop, $\maxstep{}{\oc}:=\max_{\inode\in\node}\maxstep{\inode}{\oc}$ and for an arbitrary contiguous set of outer iterations $\dfset{}\subset\until{\threshold}$, we define $\maxstep{}{\dfset{}}:=\max_{\oc\in\dfset{}}\maxstep{}{\oc}$.

Note that Assumption~\ref{asmp:data} is a standing assumption for all the results presented here. Before proceeding to analyze the properties of the iterates, we first present a short result which states that strict feasibility, Assumption~\ref{ad1}, implies a lower bound on the gradient of the constraint function over a particular level set. 
\begin{lemma}\longthmtitle{Strict feasibility implies gradient lower bound for constraint function}\label{lem:strict}
    Let Assumption~\ref{ad1} hold.  Then, there exists $G_0 > 0$ such that $\min_{\dv{}{}\in\lset{\confunc{}}{0}}\norm{\gradg{}{\dv{}{},\uv{*}{}(\dv{}{})}}\geq\mingrad$, for all $\uv{*}{}(\dv{}{})\in \mathbb{U}^*(x) :=\argmax_{\uv{}{}\in\uvset{}} \confunc{}{\dv{}{},\uv{}{}}$.
\end{lemma}
\begin{proof}
    Denote $\bar{f}(x) := \max_{u \in \mathbb{U}} f(x,u)$ and let $\bar{x} \in \mathbb{F}$ be the point satisfying $\bar{f}(\bar{x}) < 0$. By Danskin's theorem~\cite[Proposition B.25]{DPB:99}, we deduce that $\bar{f}$ is convex and 
    \begin{align}\label{eq:sub}
       \partial \bar{f}(x) = \mathrm{conv}\{\nabla_x f(x,u^*(x)) \, : \, u^*(x) \in \mathbb{U}^*(x) \} 
    \end{align}
    for any $x \in \mathbb{X}$. Now pick any point $y \in \mathbb{L}_0(f)$ and any $\zeta_y \in \partial \bar{f}(y)$. By convexity and the definition of the subgradient, we get $\bar{f}(y) + \zeta_y^\top (\bar{x} - y) \le \bar{f}(\bar{x}) < 0$. This implies that $\norm{\zeta_y} \not = 0$ and since $\zeta_y$ was arbitrarily picked, we conclude from~\eqref{eq:sub} that for any $y \in \mathbb{L}_0(f)$, we have
    \begin{align}\label{eq:grad-nonzero}
        \norm{\nabla_x f(y, u^*(y))} \not = 0, \text{ for all } u^*(y) \in \mathbb{U}^*(y).
    \end{align}
    Having established the above fact, we proceed to show the lower bound on this gradient over the set $\mathbb{L}_0(f)$. Firstly, note that $\mathbb{L}_0(f)$ is compact as $\mathbb{L}_0(f)$ is closed and is a subset of a compact set $\mathbb{X}$.  Further, from~\cite[Proposition 4.4]{JB-AS:00}, the map $x \mapsto \mathbb{U}^*(x)$ is upper semicontinuous. This means that for any sequence $\{(x_k,u_k)\}_{k \in \mathbb{N}} \subset \mathbb{X} \times \mathbb{U}$ with $u_k \in \mathbb{U}^*(x_k)$ for all $k$, if $x_k \to x'$, then all accumulation points of $\{u_k\}$ belong to $\mathbb{U}^*(x')$. We will use these properties to infer that the map $y \mapsto \kappa(y) := \min_{u^*(y) \in \mathbb{U}^*(y)} \norm{\nabla_x f(y,u^*(y))}$ is lower semicontinuous. To see the reasoning, note that since $f$ is continuously differentiable, the value $\kappa(y)$ is attained at all $y$. Thus, for any sequence $\{y_k\}$, there exists a sequence $\{u_k \in \mathbb{U}^*(y_k)\}$ such that $\kappa(y_k) = \norm{\nabla_x f(y_k,u_k)}$ for all $k$. 
    Let $y_k \to y'$ and evaluating the limit we get $\lim_{k \to \infty} \kappa(y_k) = \lim_{k \to \infty} \norm{\nabla_x f(y_k,u_k)} = \norm{\nabla_x f(y',u')}$ for some $u' \in \mathbb{U}^*(y')$. The existence of such a $u'$ is possible since the sequence belongs to a compact set. Thus, we get $\lim_{k \to \infty} \kappa(y_k) \ge \kappa(y')$ showing lower semicontinuity of $\kappa$. Finally using this property, the fact~\eqref{eq:grad-nonzero}, and compactness of $\mathbb{L}_0(f)$, we conclude the result. 
\end{proof}
We henceforth define $G_0$ as the lower bound established in the above result.
The following lemma establishes an upper bound on the distance between the current weighted average $\cdv{\oc}{\inode}$ and the next estimate $\zdv{\oc}{\inode}$ at a node $\inode\in\node$.
\begin{lemma}
\longthmtitle{Maximum distance between the weighted average and the next estimate at a node}
\label{lem:loc_consensus}
At any node $\inode\in\node$, and the $\oc^{\text{th}}$ outer iteration of Algorithm~\ref{algo:dis_com} we have
\begin{equation}
\label{eq:stepbound}
\norm{\cdv{\oc}{\inode}- \zdv{\oc}{\inode}} \leq \lrate{\oc} \norm{\gradf{\inode}{\cdv{\oc}{\inode}}}_{*}.
\end{equation}
\end{lemma}
\begin{proof}
The projected subgradient descent iteration at Step 4 of Algorithm~\ref{algo:dis_com} can be equivalently written as,
\begin{equation*}
\zdv{\oc}{\inode} = \argmin_{\dv{}{}\in\dvset{}} \left\{\objfunc{\inode}{\cdv{\oc}{\inode}}+\langle\gradf{\inode}{\cdv{\oc}{\inode}},\dv{}{}-\cdv{\oc}{\inode}\rangle+\frac{1}{2\lrate{\oc}}\norm{\dv{}{}- \cdv{\oc}{\inode}}^2\right\}.
\end{equation*}
Then, by the first-order optimality condition we have,
\begin{equation*}
0\in\lrate{\oc}\gradf{\inode}{\cdv{\oc}{\inode}} + \zdv{\oc}{\inode} - \cdv{\oc}{\inode} + \nfset{\dvset{}}\pbrack{\zdv{\oc}{\inode}},
\end{equation*}
where $\nfset{\dvset{}}\pbrack{\zdv{\oc}{\inode}}$ is the normal cone of $\dvset{}$ at $\zdv{\oc}{\inode}$. Rearranging terms and using the definition of the normal cone gives
\begin{equation*} 
\langle \lrate{\oc} \gradf{\inode}{\cdv{\oc}{\inode}}, \dumx{}{}-\zdv{\oc}{\inode}\rangle \geq \langle \cdv{\oc}{\inode}-\zdv{\oc}{\inode} , \dumx{}{}-\zdv{\oc}{\inode} \rangle, 
\end{equation*}
where $\dumx{}{}$ is any point in the domain $\dvset{}$. Let $\dumx{}{}=\cdv{\oc}{\inode}$, then by using Holder's inequality on the left-hand side, we get,
\begin{equation*} 
\lrate{\oc} \norm{\gradf{\inode}{\cdv{\oc}{\inode}}}_{*} \norm{\cdv{\oc}{\inode}-\zdv{\oc}{\inode}} \geq \norm{\cdv{\oc}{\inode}-\zdv{\oc}{\inode}}^2,
\end{equation*}
implying,
\begin{equation*}
\norm{\cdv{\oc}{\inode}-\zdv{\oc}{\inode}} \leq\lrate{\oc} \norm{\gradf{\inode}{\cdv{\oc}{\inode}}}_{*}. 
\end{equation*}
\end{proof}

As stated earlier, the role of the inner descent is to provide a feasible estimate when a node encounters infeasibility. We next proceed to establish that the inner descent loop terminates after a finite number of iterations. To this end, we first show that the closed ball used to define the inner descent step has a nonempty intersection with the feasible set of problem~\eqref{eq:sip_central}.
\begin{lemma}
\longthmtitle{Significance of the ball around an infeasible estimate}
\label{lem:ball}
For an arbitrary iteration $\oc$ of Algorithm~\ref{algo:dis_com}, we have $\ball{\oc}{\inode} \cap \F \not = \emptyset$, where $\ball{\oc}{\inode}$ is defined in Step 7 of the algorithm. 
\end{lemma}
\begin{proof}
Let $\bconfunc\pbrack{\dv{}{}}:=\max_{\uv{}{}{}\in\uvset{}}\confunc{}{\dv{}{}, \uv{}{}}$. Since $\uvset{}$ is compact and $\confunc{}$ is convex and continuously differentiable for a fixed second variable, by Danskin's theorem~\cite[Proposition B.25]{DPB:99} we deduce that $\bconfunc$ is convex and subdifferentiable. Using this fact, we deduce that $\F$ is closed and convex. Consider $\pcdv{\oc}{\inode}:=\proj{\F}{\cdv{\oc}{\inode}}$, then the vector  $\cdv{\oc}{\inode} - \pcdv{\oc}{\inode}$  belongs to the normal cone $\nfset{\F}\pbrack{\pcdv{\oc}{\inode}}$. Furthermore, according to~\cite[Theorem 2.4.7 - Corollary 1]{FHC:83}, the cone generated by the subgradient   $\partial \bconfunc\pbrack{\pcdv{\oc}{\inode}}$ satisfies  $\bigcup_{\lambda\geq 0}\lambda\partial \bconfunc\pbrack{\pcdv{\oc}{\inode}} = \nfset{\F}\pbrack{\pcdv{\oc}{\inode}}$. Pick $c\in\partial \bconfunc\pbrack{\pcdv{\oc}{\inode}}$ such that $c$ and the line segment $\cdv{\oc}{\inode}-\pcdv{\oc}{\inode}$ are collinear. Using convexity of $\bconfunc$ it follows that $\langle c, \cdv{\oc}{\inode}-\pcdv{\oc}{\inode}\rangle \leq \bconfunc\pbrack{\cdv{\oc}{\inode}}-\bconfunc\pbrack{\pcdv{\oc}{\inode}}$, along with collinearity this implies $\norm{c}\norm{\cdv{\oc}{\inode}-\pcdv{\oc}{\inode}} \leq \bconfunc\pbrack{\cdv{\oc}{\inode}}-\bconfunc\pbrack{\pcdv{\oc}{\inode}}$. We know that all nodes $\inode\in\node$ satisfy $\max_{\uv{}{}\in\uvset{}} \confunc{}{\dv{\oc}{\inode},\uv{}{}} \leq \errseq{\oc}$ implying using convexity that this property is also inherited by the weighted averages, namely, $\max_{\uv{}{}\in\uvset{}} \confunc{}{\cdv{\oc}{\inode},\uv{}{}} \leq \errseq{\oc}$. Using this fact, the definition of $\F$, and the lower bound on $\norm{c}$ due to  Lemma~\ref{lem:strict}, we have $\norm{\cdv{\oc}{\inode}-\pcdv{\oc}{\inode}}\leq \frac{\errseq{\oc}}{\mingrad}$. Subsequently, by using Assumption~\ref{ad2} and Lemma~\ref{lem:loc_consensus} we have $\norm{\cdv{\oc}{\inode}- \zdv{\oc}{\inode}} \leq \lrate{\oc} \dvgradf$. Thus, $\norm{\pcdv{\oc}{\inode}- \zdv{\oc}{\inode}} \leq \norm{\pcdv{\oc}{\inode}- \cdv{\oc}{\inode}} + \norm{\cdv{\oc}{\inode}- \zdv{\oc}{\inode}} \leq \lrate{\oc}\dvgradf+\frac{\errseq{\oc}}{\mingrad}$, implying that if we take a ball of radius $\lrate{\oc}\dvgradf+\frac{\errseq{\oc}}{\mingrad}$ centered at $\zdv{\oc}{\inode}$, it intersects with the feasible set $\F$.
\end{proof}

The following proposition establishes an upper bound on the number of inner iterations required at each node.
\begin{proposition}
\longthmtitle{Number of iterations in the inner descent loop is finite}
\label{prop:maxstep}
At a node $\inode\in\node$ and at the $\oc^{\text{th}}$ outer iteration of Algorithm~\ref{algo:dis_com}, we have $\maxstep{\inode}{\oc}\leq \frac{\dvgradg^2}{\errseq{\oc+1}^2}\pbrack{\lrate{\oc}\dvgradf+\frac{\errseq{\oc}}{\mingrad}}^2$, where $\maxstep{\inode}{\oc}$ is the number of inner iterations (Steps 8-13) for the node $\inode$ at an outer iteration $\oc$.
\end{proposition}
\begin{proof}
For an arbitrary $\ic^{\text{th}}$ inner iteration of $\oc^{\text{th}}$ outer iteration, by adapting the inequality~\cite[Lemma 8.11]{AB:17} the projected gradient descent (Step 11 in Algorithm~\ref{algo:dis_com}) satisfies,
\begin{align*}
&2\loclrate{\inode}{\oc}{\ic}\langle \gradg{}{\bzdv{\inode}{\oc}{\ic}, \buv{\inode}{\oc}{\ic}}, \bzdv{\inode}{\oc}{\ic} - \dumx{}{} \rangle \leq \norm{\dumx{}{}-\bzdv{\inode}{\oc}{\ic}}^2 \\
&\hspace{50mm} -\norm{\dumx{}{}- \bzdv{\inode}{\oc}{\ic+1}}^2 + \pbrack{\loclrate{\inode}{\oc}{\ic}}^2 \norm{\gradg{}{\bzdv{\inode}{\oc}{\ic}, \buv{\inode}{\oc}{\ic}}}^{2},
\end{align*}
for all $\dumx{}{}\in\dvset{}$. Using the first-order convexity condition of $\confunc{}$ in the first variable and the fact that for $\dumx{}{}\in\F$ we have $\confunc{}{\dumx{}{},\uv{}{}}\leq 0$ for all $\uv{}{}\in\uvset{}$, we get
\begin{align*}
&2\loclrate{\inode}{\oc}{\ic}\confunc{}{\bzdv{\inode}{\oc}{\ic},\buv{\inode}{\oc}{\ic}}\leq \norm{\dumx{}{}-\bzdv{\inode}{\oc}{\ic}}^2  -\norm{\dumx{}{}- \bzdv{\inode}{\oc}{\ic+1}}^2 + \pbrack{\loclrate{\inode}{\oc}{\ic}}^2 \norm{\gradg{}{\bzdv{\inode}{\oc}{\ic}, \buv{\inode}{\oc}{\ic}}}^{2}.
\end{align*}
Substituting the value of $\loclrate{\inode}{\oc}{\ic}$ and rearranging the terms gives,
\begin{equation}
\label{eq:inner:itr}
\frac{\confunc{}{\bzdv{\inode}{\oc}{\ic},\buv{\inode}{\oc}{\ic}}^2}{\norm{\gradg{}{\bzdv{\inode}{\oc}{\ic}, \buv{\inode}{\oc}{\ic}}}^2}\leq \norm{\dumx{}{}-\bzdv{\inode}{\oc}{\ic}}^2  -\norm{\dumx{}{}- \bzdv{\inode}{\oc}{\ic+1}}^2.
\end{equation}
Summing the above inequality for all $\ic\in\until{\maxstep{\inode}{\oc}}$ and using the bound $\norm{\gradg{}{\bzdv{\inode}{\oc}{\ic}, \buv{\inode}{\oc}{\ic}}}\leq\dvgradg$ we obtain,
\begin{align*}
\begin{split}
\frac{1}{\dvgradg^2}\sum_{\ic\in\until{\maxstep{\inode}{\oc}}}\confunc{}{\bzdv{\inode}{\oc}{\ic},\buv{\inode}{\oc}{\ic}}^2\leq \norm{\dumx{}{}-\bzdv{\inode}{\oc}{1}}^2  -\norm{\dumx{}{}- \bzdv{\inode}{\oc}{\maxstep{\inode}{\oc}+1}}^2,
\end{split}
\end{align*}
implying,
\begin{align*}
\begin{split}
\maxstep{\inode}{\oc}\min_{\ic\in\until{\maxstep{\inode}{\oc}}}\confunc{}{\bzdv{\inode}{\oc}{\ic},\buv{\inode}{\oc}{\ic}}^2 &\overset{(a)}{\leq} \dvgradg^2 \norm{\dumx{}{}-\bzdv{\inode}{\oc}{1}}^2 \text{, and so,}\\
\min_{\ic\in\until{\maxstep{\inode}{\oc}}}\confunc{}{\bzdv{\inode}{\oc}{\ic},\buv{\inode}{\oc}{\ic}} &\leq \frac{\dvgradg \norm{\dumx{}{}-\bzdv{\inode}{\oc}{1}}}{\sqrt{\maxstep{\inode}{\oc}}},
\end{split}
\end{align*}
where (a) follows by replacing all terms on the left-hand side with the minimum value and removing the negative term from the right-hand side. Recall that for inner descent iterations to terminate at an arbitrary inner iteration $\bar{\ic}$, we require that $\confunc{}{\bzdv{\inode}{\oc}{\bar{\ic}},\buv{\inode}{\oc}{\bar{\ic}}}\leq \errseq{\oc+1}$. Using Lemma~\ref{lem:ball} we know that there exists a $\bar{\dumx{}{}}\in\F$ such that $\norm{\bar{\dumx{}{}}-\bzdv{\inode}{\oc}{1}}\leq\lrate{\oc}\dvgradf+\frac{\errseq{\oc}}{\mingrad}$, therefore by substituting $\dumx{}{}=\bar{\dumx{}{}}$ in the above inequality and upper bounding the right-hand side we get, $\min_{\ic\in\until{\maxstep{\inode}{\oc}}}\confunc{}{\bzdv{\inode}{\oc}{\ic},\buv{\inode}{\oc}{\ic}}\leq \errseq{\oc+1}$, where this  inequality follows from the lower bound for $\maxstep{\inode}{\oc}$ hypothesised in the statement.
\end{proof}

Lemma~\ref{lem:ball} and Proposition~\ref{prop:maxstep} established properties that depict the local behaviour of the algorithm at any given node. However, since the algorithm is distributed, its analysis requires us to consider the interaction between various nodes. This interaction in our algorithm takes place in Step 3 where nodes build a weighted average of the estimates of the their neighbors. The next task is to combine the preceding steps to demonstrate that our algorithm points in the direction of one of the optimizers of problem~\eqref{eq:sip_central} on average. 
\begin{lemma}
\longthmtitle{Bound on the cumulative inner product between the subgradient and the vector joining an optimizer to the estimate}
\label{lem:sum_grad}
For any $\dv{*}{}\in\solset$ and an arbitrary contiguous set of outer iterations $\dfset{}:=\setdef{\dthres, \dthres+1,\cdots,\kthres}$, the sequence $\seqdef{\cdv{\oc}{\inode}}{\oc\in\dfset{}}$ generated via Algorithm~\ref{algo:dis_com} at any node $\inode \in \node$ satisfies:
\begin{align}
\label{eq:sum_grad}
\begin{split}
&\sum_{\oc\in\dfset{}}\lrate{\oc} \sum_{\inode\in\node}\langle \gradf{\inode}{\cdv{\oc}{\inode}}, \cdv{\oc}{\inode} - \dv{*}{}\rangle \leq\frac{\tnode \setdia}{ 2}
+ \frac{\tnode\maxgrad^2}{2} \sum_{\oc\in\dfset{}} \pbrack{\lrate{\oc}^2 + \frac{\pbrack{\errseq{\oc}+\lrate{\oc}\dvgradf\mingrad}^2}{\mingrad^4}}.
\end{split}
\end{align}
\end{lemma}
\begin{proof}
For an arbitrary $\oc^{\text{th}}$ outer iteration, adapting the fundamental inequality~\cite[Lemma 8.11]{AB:17} of projected subgradient descent (Step 4 in Algorithm~\ref{algo:dis_com}) to our case yields 
\begin{align}
\label{eq:pgd1}
\begin{split}
&2\lrate{\oc} \langle \gradf{\inode}{\cdv{\oc}{\inode}}, \cdv{\oc}{\inode} - \dumx{}{}\rangle \leq \norm{\dumx{}{} - \cdv{\oc}{\inode}}^2 - \norm{\dumx{}{} - \zdv{\oc}{\inode}}^2 + \lrate{\oc}^2 \norm{\gradf{\inode}{\cdv{\oc}{\inode}}}^2,  
\end{split}
\end{align}
for any $\dumx{}{}\in\F$. Similarly, for an arbitrary $\ic^{\text{th}}$ inner iteration of $\oc^{\text{th}}$ outer iteration, the projected gradient descent (Step 10 in Algorithm~\ref{algo:dis_com}) satisfies,
\begin{align}
\label{eq:pgd2}
\begin{split}
&2\loclrate{\inode}{\oc}{\ic}\langle \gradg{}{\bzdv{\inode}{\oc}{\ic}, \buv{\inode}{\oc}{\ic}}, \bzdv{\inode}{\oc}{\ic} - \dumx{}{} \rangle \leq \norm{\dumx{}{}-\bzdv{\inode}{\oc}{\ic}}^2 \\
&\hspace{45mm} -\norm{\dumx{}{}- \bzdv{\inode}{\oc}{\ic+1}}^2 + \pbrack{\loclrate{\inode}{\oc}{\ic}}^2 \norm{\gradg{}{\bzdv{\inode}{\oc}{\ic}, \buv{\inode}{\oc}{\ic}}}^{2},
\end{split}
\end{align}
for any $\dumx{}{}\in\F$. Now, by summing~\eqref{eq:pgd2} (and telescoping the norms) over all the inner iterations, i.e., for all $\ic\in\until{\maxstep{\inode}{\oc}}$ and adding this sum to~\eqref{eq:pgd1}, we get,
\begin{align*}
&2\lrate{\oc} \langle \gradf{\inode}{\cdv{\oc}{\inode}}, \cdv{\oc}{\inode} - \dumx{}{}\rangle +\sum_{\ic\in\until{\maxstep{\inode}{\oc}}}2\loclrate{\inode}{\oc}{\ic}\langle \gradg{}{\bzdv{\inode}{\oc}{\ic}, \buv{\inode}{\oc}{\ic}}, \bzdv{\inode}{\oc}{\ic} - \dumx{}{}\rangle \\ 
&\hspace{5mm}\leq \norm{\dumx{}{}-\cdv{\oc}{\inode}}^2 -\norm{\dumx{}{}- \bzdv{\inode}{\oc}{\maxstep{\inode}{\oc}+1}}^2  +\lrate{\oc}^2 \norm{\gradf{\inode}{\cdv{\oc}{\inode}}}^2 + \sum_{\ic\in\until{\maxstep{\inode}{\oc}}}\pbrack{\loclrate{\inode}{\oc}{\ic}}^2 \norm{\gradg{}{\bzdv{\inode}{\oc}{\ic}, \buv{\inode}{\oc}{\ic}}}^{2}.
\end{align*}
Consider the term $\langle \gradg{}{\bzdv{\inode}{\oc}{\ic}, \buv{\inode}{\oc}{\ic}}, \bzdv{\inode}{\oc}{\ic} - \dumx{}{}\rangle$. Using convexity of $\confunc{}$ in the first variable, we have $\langle \gradg{}{\bzdv{\inode}{\oc}{\ic}, \buv{\inode}{\oc}{\ic}}, \bzdv{\inode}{\oc}{\ic} - \dumx{}{}\rangle \geq \confunc{}{\bzdv{\inode}{\oc}{\ic},\buv{\inode}{\oc}{\ic}} - \confunc{}{\dumx{}{},\buv{\inode}{\oc}{\ic}} \geq \confunc{}{\bzdv{\inode}{\oc}{\ic},\buv{\inode}{\oc}{\ic}}\geq 0$, because $\bzdv{\inode}{\oc}{\ic}\notin \F$ and $\dumx{}{}\in\F$. Therefore, by removing the non-negative terms from the left-hand side we have
\begin{align*}
&2\lrate{\oc}\langle \gradf{\inode}{\cdv{\oc}{\inode}}, \cdv{\oc}{\inode} - \dumx{}{}\rangle \\
&\leq \norm{\dumx{}{}- \cdv{\oc}{\inode}}^2 -\norm{\dumx{}{}- \bzdv{\inode}{\oc}{\maxstep{\inode}{\oc}+1}}^2 +\lrate{\oc}^2 \norm{\gradf{\inode}{\cdv{\oc}{\inode}}}^2 + \sum_{\ic\in\until{\maxstep{\inode}{\oc}}}\pbrack{\loclrate{\inode}{\oc}{\ic}}^2 \norm{\gradg{}{\bzdv{\inode}{\oc}{\ic}, \buv{\inode}{\oc}{\ic}}}^{2} \\
&\overset{(a)}{\leq} \normb{\dumx{}{}- \sum_{\jnode\in\node}\wtgraph{\oc}{\inode}{\jnode} \dv{\oc}{\jnode}}^2 -\norm{\dumx{}{}- \dv{\oc+1}{\inode}}^2+\lrate{\oc}^2 \dvgradf^2 + \sum_{\ic\in\until{\maxstep{\inode}{\oc}}}\pbrack{\loclrate{\inode}{\oc}{\ic}}^2 \dvgradg^{2} \\
&\overset{(b)}{\leq} \sum_{\jnode\in\node}\wtgraph{\oc}{\inode}{\jnode} \norm{\dumx{}{}- \dv{\oc}{\jnode}}^2 -\norm{\dumx{}{}- \dv{\oc+1}{\inode}}^2+\lrate{\oc}^2 \dvgradf^2 + \sum_{\ic\in\until{\maxstep{\inode}{\oc}}}\pbrack{\loclrate{\inode}{\oc}{\ic}}^2 \dvgradg^{2}, 
\end{align*}
where (a) follows from the bounds on the gradients and definitions for different terms in Algorithm~\ref{algo:dis_com}, and (b) follows from the convexity of the square of the norm. Let $\dumx{}{}=\dv{*}{}\in\solset$, then by summing over all the nodes in $\node$ we get,
\begin{align*}
&2\sum_{\inode\in\node}\lrate{\oc}\langle \gradf{\inode}{\cdv{\oc}{\inode}}, \cdv{\oc}{\inode} - \dv{*}{}\rangle \\
&\leq \sum_{\inode\in\node}\pbrack{\sum_{\jnode\in\node}\wtgraph{\oc}{\inode}{\jnode} \norm{\dv{*}{}- \dv{\oc}{\jnode}}^2 -\norm{\dv{*}{}- \dv{\oc+1}{\inode}}^2} +\sum_{\inode\in\node}\pbrack{\lrate{\oc}^2 \dvgradf^2 + \sum_{\ic\in\until{\maxstep{\inode}{\oc}}}\pbrack{\loclrate{\inode}{\oc}{\ic}}^2 \dvgradg^{2}}  \\
&= \sum_{\jnode\in\node}\sum_{\inode\in\node}\wtgraph{\oc}{\inode}{\jnode} \norm{\dv{*}{}- \dv{\oc}{\jnode}}^2 -\sum_{\jnode\in\node}\norm{\dv{*}{}- \dv{\oc+1}{\jnode}}^2+\tnode\lrate{\oc}^2 \dvgradf^2 + \tnode\sum_{\ic\in\until{\maxstep{\inode}{\oc}}}\pbrack{\loclrate{\inode}{\oc}{\ic}}^2 \dvgradg^{2} \\
&\overset{(c)}{\leq} \sum_{\jnode\in\node} \pbrack{\norm{\dv{*}{}- \dv{\oc}{\jnode}}^2 -\norm{\dv{*}{}- \dv{\oc+1}{\jnode}}^2} +\tnode\lrate{\oc}^2 \dvgradf^2 + \tnode\sum_{\ic\in\until{\maxstep{\inode}{\oc}}}\pbrack{\loclrate{\inode}{\oc}{\ic}}^2 \dvgradg^{2}, 
\end{align*}
where (c) is due to the double stochastic nature of the weight matrix (i.e., $\sum_{\inode\in\node}\wtgraph{\oc}{\inode}{\jnode}=1$). Now sum over $\oc\in\dfset{}$ and obtain
\begin{align*}
&2\sum_{\oc\in\dfset{}} \sum_{\inode\in\node}\lrate{\oc}\langle \gradf{\inode}{\cdv{\oc}{\inode}}, \cdv{\oc}{\inode} - \dv{*}{}\rangle  \\
&\leq \sum_{\oc\in\dfset{}} \sum_{\inode\in\node} \pbrack{\norm{\dv{*}{}- \dv{\oc}{\inode}}^2 -\norm{\dv{*}{}- \dv{\oc+1}{\inode}}^2}+ \tnode\sum_{\oc\in\dfset{}} \pbrack{\lrate{\oc}^2 \dvgradf^2 + \sum_{\ic\in\until{\maxstep{\inode}{\oc}}}\pbrack{\loclrate{\inode}{\oc}{\ic}}^2 \dvgradg^{2}}  \\
&\leq \sum_{\inode\in\node}\pbrack{\norm{\dv{*}{}- \dv{\dthres}{\inode}}^2 -\norm{\dv{*}{}- \dv{\kthres+1}{\inode}}^2} + \tnode\maxgrad^2\sum_{\oc\in\dfset{}} \pbrack{\lrate{\oc}^2 + \sum_{\ic\in\until{\maxstep{\inode}{\oc}}}\frac{\confunc{}{\bzdv{\inode}{\oc}{\ic},\buv{\inode}{\oc}{\ic}}^2}{\norm{\gradg{}{\bzdv{\inode}{\oc}{\ic}, \buv{\inode}{\oc}{\ic}}}^4}}. \\
\end{align*}
Using the bound on the gradient of the constraint function and~\eqref{eq:inner:itr}, we can write the following for any $\dumx{}{}\in\F$
\begin{align*}
2\sum_{\oc\in\dfset{}} \sum_{\inode\in\node}&\lrate{\oc}\langle \gradf{\inode}{\cdv{\oc}{\inode}}, \cdv{\oc}{\inode} - \dv{*}{}\rangle \\
&\hspace{0mm}\leq \tnode \setdia + \tnode\maxgrad^2\sum_{\oc\in\dfset{}} \biggl(\lrate{\oc}^2 +\frac{1} {\mingrad^2}\sum_{\ic\in\until{\maxstep{\inode}{\oc}}} \left(\norm{\dumx{}{}-\bzdv{\inode}{\oc}{\ic}}^2-\norm{\dumx{}{}- \bzdv{\inode}{\oc}{\ic+1}}^2\right)\biggr),\\
&\hspace{0mm}\overset{(d)}{\leq} \tnode \setdia + \tnode\maxgrad^2\sum_{\oc\in\dfset{}} \biggl(\lrate{\oc}^2 +\frac{1} {\mingrad^2} \left(\norm{\dumx{}{}-\bzdv{\inode}{\oc}{1}}^2-\norm{\dumx{}{}- \bzdv{\inode}{\oc}{\maxstep{\inode}{\oc}+1}}^2\right)\biggr),\\
&\hspace{0mm}\overset{(e)}{\leq} \tnode \setdia + \tnode\maxgrad^2\sum_{\oc\in\dfset{}} \biggl(\lrate{\oc}^2 +\frac{1} {\mingrad^2} \norm{\dumx{}{}-\bzdv{\inode}{\oc}{1}}^2\biggr),\\
\end{align*}
where (d) follows by telescoping the sums and (e) follows by removing the negative norm term from the upper bound. Using Lemma~\ref{lem:ball} we know that there exists a $\bar{\dumx{}{}}\in\F$ such that $\norm{\bar{\dumx{}{}}-\bzdv{\inode}{\oc}{1}}\leq\lrate{\oc}\dvgradf+\frac{\errseq{\oc}}{\mingrad}$, therefore by substituting $\dumx{}{}=\bar{\dumx{}{}}$ in the above inequality and upper bounding the right-hand side we get the required bound.
\end{proof}
The primary objective of our convergence analysis is to demonstrate that all nodes converge to the same estimate of the optimizer of problem~\eqref{eq:sip_central}; that is, they achieve consensus on this value. Consensus among nodes of a network is a well-studied topic, and we draw upon the standard results from~\cite{AN-AO:09}, adapting them to analyze our approach. In particular, we rely on the following lemma based on~\cite[Proposition 1]{AN-AO:09} as a basis to argue that estimates at different nodes of the network achieve consensus. Under the assumed properties of the network, this result provides a bound that characterizes the convergence behaviour of the entries of the network transition matrix.
\begin{lemma}
\longthmtitle{Convergence of network transition matrix~\cite[Lemma 1]{JL-GL-ZW-CW:18}}
\label{lem:graph_flow}
Let Assumption~\ref{asmp:graph} hold. Then, for all iterations $t > s \geq 1$ and all nodes $\inode, \jnode \in \node$, the following inequality holds:
\begin{equation}
\label{eq:tranmat}
    \abs{\until{\graphflow{t}{s}}_{\inode\jnode}-\frac{1}{\tnode}}\leq \nodepar \gphpar^{t-s},
\end{equation}
where $\nodepar = \pbrack{1-\frac{\minwt}{4\tnode^2}}^{-2}$ and $\gphpar=\pbrack{1-\frac{\minwt}{4\tnode^2}}^{\frac{1}{\strong}}$.
\end{lemma}
To track the progression of consensus among the estimates at each node, we define the following network-wide average of the estimates 
\begin{equation}\label{eq:avg-estimate}
    \hdv{\oc}{}:=\frac{1}{\tnode}\sum_{\inode\in\node}\dv{\oc}{\inode},
\end{equation}
and analyze how individual estimates evolve relative to this average. We next present Lemma~\ref{lem:consensus}, which establishes an upper bound on the difference between the network-wide average of estimates of the optimizer (i.e., $\hdv{\oc}{}$) and the estimate at an arbitrary node. This bound serves as a measure of consensus across the network, and it vanishes as the number of outer iterations increases, indicating convergence of all estimates to a common value.

\begin{lemma}
\longthmtitle{Bound on the disagreement of a node's estimate with the network-wide average of the estimate}
\label{lem:consensus}
At any node $\inode\in\node$ and at the $\oc^{\text{th}}$ outer iteration of Algorithm~\ref{algo:dis_com}, we have
\begin{equation}
\label{eq:consensus}
    \norm{\hdv{\oc+1}{}-\dv{\oc+1}{\inode}} \leq \conbound{\oc}, 
\end{equation}
where, 
\begin{align*}
&\conbound{\oc}:=\tnode\nodepar\sum_{s=1}^{\oc-1} \gphpar^{\oc-s-1}\pbrack{2\lrate{s} \dvgradf + \frac{\errseq{s}}{\mingrad}} + \nodepar \gphpar^{\oc-1} \sum_{\jnode\in\node} \norm{\dv{1}{\jnode}} + 2\pbrack{2\lrate{\oc} \dvgradf + \frac{\errseq{\oc}}{\mingrad}},
\end{align*}
with, $\nodepar = \pbrack{1-\frac{\minwt}{4\tnode^2}}^{-2}$ and $\gphpar=\pbrack{1-\frac{\minwt}{4\tnode^2}}^{\frac{1}{\strong}}$. As a consequence, we have \(\norm{\dv{\oc+1}{\jnode}-\dv{\oc+1}{\inode}}\leq 2\conbound{\oc}\).
\end{lemma}
\begin{proof}
The proof is based on~\cite[Lemma 4]{JL-GL-ZW-CW:18}. The following recursion can be obtained for Step 3 of Algorithm~\ref{algo:dis_com} using the network transition matrix,
\begin{align*}
\cdv{\oc}{\inode} = \sum_{\jnode\in\node} \until{\graphflow{\oc}{1}}_{\inode\jnode}\dv{1}{\jnode} + \sum_{s=1}^{k-1}\sum_{\jnode\in\node} \until{\graphflow{\oc}{s+1}}_{\inode\jnode}(\dv{s+1}{\inode}-\cdv{s}{\inode}).
\end{align*}
By adding $\dv{\oc+1}{\inode}$ to both sides, we write
\begin{align}
\label{eq:consensus1}
&\dv{\oc+1}{\inode} = \sum_{\jnode\in\node} \until{\graphflow{\oc}{1}}_{\inode\jnode}\dv{1}{\jnode} +\sum_{s=1}^{k-1}\sum_{\jnode\in\node} \until{\graphflow{\oc}{s+1}}_{\inode\jnode}(\dv{s+1}{\inode}-\cdv{s}{\inode})+\dv{\oc+1}{\inode}-\cdv{\oc}{\inode}.
\end{align}
Using the above equality and the fact that adjacency matrix is doubly stochastic at all times, we write the network-wide average $\hdv{\oc+1}{}$ as
\begin{align}
\label{eq:consensus2}
\begin{split}
\hdv{\oc+1}{} &= \frac{1}{\tnode}\sum_{\inode\in\node} \dv{\oc+1}{\inode}\\
&=\frac{1}{\tnode}\sum_{\jnode\in\node}\dv{1}{\jnode} + \frac{1}{\tnode}\sum_{s=1}^{k-1}\sum_{\jnode\in\node}(\dv{s+1}{\inode}-\cdv{s}{\inode})+\frac{1}{\tnode}\sum_{\inode\in\node}(\dv{\oc+1}{\inode}-\cdv{\oc}{\inode}).
\end{split}
\end{align}
Now by subtracting~\eqref{eq:consensus1} from~\eqref{eq:consensus2} and taking norm on both sides, we get the following using Holder's inequality:
\begin{align*}
&\norm{\hdv{\oc+1}{}-\dv{\oc+1}{\inode}} \leq \sum_{\jnode\in\node} \abs{\until{\graphflow{\oc}{1}}_{\inode\jnode}-\frac{1}{\tnode}}\norm{\dv{1}{\jnode}} + \sum_{s=1}^{\oc-1}\sum_{\jnode\in\node} \abs{\until{\graphflow{\oc}{s+1}}_{\inode\jnode}-\frac{1}{\tnode}}\norm{\cdv{s}{\jnode}-\dv{s+1}{\jnode}} \\
&\hspace{52mm}+\frac{1}{\tnode}\sum_{\jnode\in\node}\norm{\cdv{\oc}{\jnode}-\dv{\oc+1}{\jnode}} + \norm{\cdv{\oc}{\inode}-\dv{\oc+1}{\inode}}.
\end{align*}
Consequently,
\begin{align*}
&\norm{\hdv{\oc+1}{}-\dv{\oc+1}{\inode}} \\
&\overset{(a)}{\leq} \nodepar \gphpar^{\oc-1} \sum_{\jnode\in\node} \norm{\dv{1}{\jnode}} + \sum_{s=1}^{\oc-1} \nodepar \gphpar^{\oc-s-1} \sum_{\jnode\in\node}\norm{\cdv{s}{\jnode}-\dv{s+1}{\jnode}} \\
&\hspace{6mm}+\frac{1}{\tnode}\sum_{\jnode\in\node}\norm{\cdv{\oc}{\jnode}-\dv{\oc+1}{\jnode}} + \norm{\cdv{\oc}{\inode}-\dv{\oc+1}{\inode}}\\
&\overset{(b)}{\leq} \sum_{s=1}^{\oc-1} \nodepar \gphpar^{\oc-s-1} \sum_{\jnode\in\node}\pbrack{\norm{\cdv{s}{\jnode}-\zdv{s}{\jnode}}+\norm{\zdv{s}{\jnode}-\dv{s+1}{\jnode}}} \\
&\hspace{6mm}+\nodepar \gphpar^{\oc-1} \sum_{\jnode\in\node} \norm{\dv{1}{\jnode}}+ \frac{1}{\tnode}\sum_{\jnode\in\node}\pbrack{\norm{\cdv{\oc}{\jnode}-\zdv{\oc}{\jnode}}+\norm{\zdv{\oc}{\jnode}-\dv{\oc+1}{\jnode}}} \\
&\hspace{6mm}+ \norm{\cdv{\oc}{\inode}-\zdv{\oc}{\inode}}+\norm{\zdv{\oc}{\inode}-\dv{\oc+1}{\inode}}\\
&\overset{(c)}{\leq}  \tnode\nodepar\sum_{s=1}^{\oc-1} \gphpar^{\oc-s-1}\pbrack{2\lrate{s} \dvgradf + \frac{\errseq{s}}{\mingrad}} + \nodepar \gphpar^{\oc-1} \sum_{\jnode\in\node} \norm{\dv{1}{\jnode}} + 2\pbrack{2\lrate{\oc} \dvgradf + \frac{\errseq{\oc}}{\mingrad}},
\end{align*}
where (a) follows via Lemma~\ref{lem:graph_flow}, (b) is obtained via the use of triangle inequality, and (c) is due to the upper bound provided in Lemma~\ref{lem:loc_consensus} and the radius of the closed ball in Lemma~\ref{lem:ball}.
\end{proof}
Using the bound provided in the above lemma, the result in~\cite[Lemma 3.1]{SSR-AN-VVV:10} and since both the stepsize and the feasibility tolerance sequence approach zero, we deduce that $\conbound{\oc} \to 0$ for $\oc\to\infty$. Consequently, for Algorithm~\ref{algo:dis_com}, we obtain $\lim_{\oc\to\infty}\norm{\hdv{\oc+1}{}-\dv{\oc+1}{\inode}}= 0$ and  $\lim_{\oc\to\infty}\norm{\dv{\oc+1}{\jnode}-\dv{\oc+1}{\inode}}= 0$. 
This concludes the preliminary results which form the foundation for the main result presented next. The following theorem summarizes the convergence and feasibility guarantees of Algorithm~\ref{algo:dis_com}.

\begin{theorem}
\longthmtitle{Optimality and feasibility guarantees of  Algorithm~\ref{algo:dis_com}}
\label{thm:optimality}
Consider the optimization problem~\eqref{eq:sip_central} along with its data and suppose Assumption~\ref{asmp:graph} and~\ref{asmp:data} hold. Then for an arbitrary node $\inode\in\node$, Algorithm~\ref{algo:dis_com} satisfies :
\begin{enumerate}
\item[\namedlabel{thm1}{(\ref{thm:optimality}-a)}] For all outer iterations $\oc\in\fset{}$, the inner descent loop terminates within a maximum of $\maxstep{}{}=2\dvgradg^2\pbrack{\setdia\dvgradf+\frac{1}{\mingrad}}^2$ iterations.
\item[\namedlabel{thm2}{(\ref{thm:optimality}-b)}] For the estimate $\bdv{\threshold}{\inode}$, we have
\begin{align}
\label{eq:optimality}
\begin{split}
\objfunc{}{\bdv{\threshold}{\inode}} &- \objfunc{}{\dv{*}{}} \leq \frac{\errop\pbrack{\gphpar,\nodepar,\dvgradf,\mingrad,\dvgradg,\maxgrad,\setdia,\tnode,\sum_{\jnode\in\node}\norm{\dv{1}{\jnode}}}}{\sqrt{\threshold}},
\end{split}
\end{align}
where $\errop\pbrack{\gphpar,\nodepar,\dvgradf,\mingrad,\dvgradg,\maxgrad,\setdia,\tnode,\sum_{\jnode\in\node}\norm{\dv{1}{\jnode}}}$ is a constant depending on the problem data and assumptions.
\item[\namedlabel{thm3}{(\ref{thm:optimality}-c)}] For the estimate $\bdv{\threshold}{\inode}$ we have $\sup_{\uv{}{}\in\uvset{}} \confunc{}{\bdv{}{\inode},\uv{}{}} \leq \frac{2\setdia\ln{2}}{(2-\sqrt{2})\sqrt{\threshold}}$.
\end{enumerate}
\end{theorem}
\begin{proof}
 The proof follows by substituting the expressions for the sequence of stepsizes and feasibility tolerances into previously established results.  We divide the proof into three parts (i), (ii) and (iii), each addressing one respective claim of the result.
\begin{enumerate}
\item[\namedlabel{thmproof1}{(\romannumeral 1)}] By substituting the expressions for $\lrate{\oc}$ and $\errseq{\oc}$ in the lower bound for the number of inner iterations established in the Proposition~\ref{prop:maxstep} and using $1+1/\oc \le 2$ for any $k \ge 1$,  we get the required bound. 
\item[\namedlabel{thmproof2}{(\romannumeral 2)}] %
For the  objective function, we have
\begin{align*}
\objfunc{}{\bdv{\threshold}{\inode}} - \objfunc{}{\dv{*}{}} &= \objfunc{}{\frac{1}{\sum_{\oc\in\fset{}}\lrate{\oc}}\sum_{\oc\in\fset{}}\lrate{\oc}\dv{\oc+1}{\inode}} - \objfunc{}{\dv{*}{}} \\
&\overset{(a)}{\leq} \frac{1}{\sum_{\oc\in\fset{}}\lrate{\oc}}  \sum_{\oc\in\fset{}}\lrate{\oc} \pbrack{\objfunc{}{\dv{\oc+1}{\inode}} - \objfunc{}{\dv{*}{}}} \\
&= \frac{1}{\sum_{\oc\in\fset{}}\lrate{\oc}}  \sum_{\oc\in\fset{}}\lrate{\oc} \sum_{\jnode\in\node}\pbrack{\objfunc{\jnode}{\dv{\oc+1}{\inode}} - \objfunc{\jnode}{\dv{*}{}}} \\
&= \frac{1}{\sum_{\oc\in\fset{}}\lrate{\oc}}  \sum_{\oc\in\fset{}}\lrate{\oc} \sum_{\jnode\in\node}\pbrack{\objfunc{\jnode}{\cdv{\oc}{\jnode}} - \objfunc{\jnode}{\dv{*}{}}}\\
&\quad+\frac{1}{\sum_{\oc\in\fset{}}\lrate{\oc}}  \sum_{\oc\in\fset{}}\lrate{\oc} \sum_{\jnode\in\node}\pbrack{\objfunc{\jnode}{\dv{\oc+1}{\inode}} - \objfunc{\jnode}{\cdv{\oc}{\jnode}}} \\
&\overset{(b)}{\leq} \underbrace{\frac{1}{\sum_{\oc\in\fset{}}\lrate{\oc}}  \sum_{\oc\in\fset{}}\lrate{\oc} \sum_{\jnode\in\node}\langle \gradf{\jnode}{\cdv{\oc}{\jnode}}, \cdv{\oc}{\jnode} - \dv{*}{}\rangle}_{A}\\
&\quad+\underbrace{\frac{1}{\sum_{\oc\in\fset{}}\lrate{\oc}}  \sum_{\oc\in\fset{}}\lrate{\oc} \sum_{\jnode\in\node} \langle \gradf{\jnode}{\dv{\oc+1}{\inode}}, \dv{\oc+1}{\inode} - \cdv{\oc}{\jnode}\rangle}_{B},
\end{align*}  
where (a) and (b) are due to the convexity of $\objfunc{}$ and $\objfunc{\jnode}$ respectively. Using Lemma~\ref{lem:sum_grad} we can write
\begin{align*}
&A \leq \frac{\tnode}{\sum_{\oc\in\fset{}} \!  \lrate{\oc}}\left(\frac{ \setdia}{ 2} \! + \! \frac{\maxgrad^2}{2} \sum_{\oc\in\fset{}} \!  \! \pbrack{\lrate{\oc}^2 + \frac{\pbrack{\errseq{\oc}+\lrate{\oc}\dvgradf\mingrad}^2}{\mingrad^4}}\right)
\\
&\overset{(c)}{\leq} \frac{\tnode}{\sum_{\oc\in\fset{}}\lrate{\oc}}\left(\frac{ \setdia}{ 2}  \!  +  \!   \frac{\maxgrad^2}{2} \pbrack{\setdia^2 \!  + \! \frac{\pbrack{\errseq{\oc}+\lrate{\oc}\dvgradf\mingrad}^2}{\mingrad^4}}\sum_{\oc\in\fset{}}\!  \frac{1}{\oc} \right)\\
&\overset{(d)}{\leq} \frac{\tnode}{\sum_{\oc\in\fset{}}\lrate{\oc}}\pbrack{\frac{ \setdia}{ 2} + \maxgrad^2 \pbrack{\setdia^2 + \frac{\pbrack{\errseq{\oc}+\lrate{\oc}\dvgradf\mingrad}^2}{\mingrad^4}} \ln{2}},
\end{align*}
where, (c) follows by substituting $\lrate{\oc} = \frac{\setdia}{\sqrt{\oc}}$ and $\errseq{\oc} = \frac{1}{\sqrt{\oc}}$ and (d) is due to the inequality $\sum_{\oc\in\fset{}}\frac{1}{\oc} = \sum_{\oc=\grtint{\frac{\threshold}{2}}}^\threshold \frac{1}{\oc} \leq 2\ln{2}$. 
For $B$, we have
\begin{align*}
B &\overset{(e)}{\leq} \frac{1}{\sum_{\oc\in\fset{}}\lrate{\oc}}\sum_{\oc\in\fset{}} \lrate{\oc}\sum_{\jnode\in\node} \norm{\gradf{\jnode}{\dv{\oc+1}{\inode}}}_{*}\norm{\dv{\oc+1}{\inode} - \cdv{\oc}{\jnode}}\\
&\overset{(f)}{=} \frac{\dvgradf}{\sum_{\oc\in\fset{}}\lrate{\oc}}\sum_{\oc\in\fset{}} \lrate{\oc}\sum_{\jnode\in\node} \norm{\dv{\oc+1}{\inode} -\dv{\oc+1}{\jnode} +\dv{\oc+1}{\jnode}-\cdv{\oc}{\jnode}}
\\
&\overset{(g)}{\leq} \frac{\dvgradf}{\sum_{\oc\in\fset{}}\lrate{\oc}} \!  \sum_{\oc\in\fset{}}\lrate{\oc}  \!  \sum_{\jnode\in\node} \pbrack{\norm{\dv{\oc+1}{\inode}\!  - \! \dv{\oc+1}{\jnode}} \!  + \! \norm{\dv{\oc+1}{\jnode} - \cdv{\oc}{\jnode}}}
\\
&\overset{(h)}{\leq} \frac{2\dvgradf\tnode}{\sum_{\oc\in\fset{}}\lrate{\oc}} \! \underbrace{\sum_{\oc\in\fset{}}\lrate{\oc}\conbound{\oc}}_{C} \! + \! \frac{\dvgradf\tnode}{\sum_{\oc\in\fset{}}\lrate{\oc}} \underbrace{\sum_{\oc\in\fset{}}\pbrack{2\lrate{\oc}^2\dvgradf \! + \! \frac{\lrate{\oc}\errseq{\oc}}{\mingrad}}}_{D}, 
\end{align*}    
where (e) is due to Holder's inequality, (f) follows from the upper bound of the gradient, (g) follows via triangle inequality, and (h) is obtained by bounding the first term in (g) (consensus among nodes) using Lemma~\ref{lem:consensus} and the second term in (g) is bounded by $\norm{\zdv{\oc}{\jnode} - \cdv{\oc}{\jnode}}$+$\norm{\dv{\oc+1}{\jnode} - \zdv{\oc}{\jnode}}$, where $\norm{\zdv{\oc}{\jnode} - \cdv{\oc}{\jnode}}$ is upper bounded using Lemma~\ref{lem:loc_consensus}
and $\norm{\dv{\oc+1}{\jnode} - \zdv{\oc}{\jnode}}$ is bounded by the radius of the closed ball provided in Lemma~\ref{lem:ball}. Now we have 
\begin{align*}
C = \sum_{\oc\in\fset{}} \lrate{\oc} \conbound{\oc}&\overset{(i)}{=} \sum_{\oc\in\fset{}} \lrate{\oc} \biggl(\tnode\nodepar\sum_{s=1}^{\oc-1} \gphpar^{\oc-s-1}\pbrack{2\lrate{s} \dvgradf + \frac{\errseq{s}}{\mingrad}}\biggr. \\
&\hspace{10mm}\biggl.+ \nodepar \gphpar^{\oc-1} \sum_{\jnode\in\node} \norm{\dv{1}{\jnode}} + 2\pbrack{2\lrate{\oc} \dvgradf + \frac{\errseq{\oc}}{\mingrad}}\biggr) \\
&= \underbrace{\tnode\nodepar\sum_{\oc\in\fset{}} \lrate{\oc} \sum_{s=1}^{\oc-1} \gphpar^{\oc-s-1}\pbrack{2\lrate{s} \dvgradf + \frac{\errseq{s}}{\mingrad}}}_{E} \\
&\hspace{10mm}+\underbrace{ \nodepar\sum_{\oc\in\fset{}} \lrate{\oc}\gphpar^{\oc-1} \sum_{\jnode\in\node} \norm{\dv{1}{\jnode}}}_{F} + \underbrace{2\sum_{\oc\in\fset{}}\lrate{\oc}\pbrack{2\lrate{\oc} \dvgradf + \frac{\errseq{\oc}}{\mingrad}}}_{G}, \\
\end{align*}
where (i) is via Lemma~\ref{lem:consensus}. For the first term above, we have
\begin{align*}
E &\overset{(j)}{=} \tnode\nodepar\sum_{\oc\in\fset{}} \frac{\setdia}{\sqrt{\oc}} \sum_{s=1}^{\oc-1} \gphpar^{\oc-s-1}\pbrack{2\frac{\setdia}{\sqrt{s}} \dvgradf + \frac{1}{\sqrt{s}}\frac{1}{\mingrad}} \\
&= \tnode\setdia\nodepar\pbrack{2\setdia \dvgradf + \frac{1}{\mingrad}}\sum_{\oc\in\fset{}} \frac{1}{\sqrt{\oc}} \sum_{s=1}^{\oc-1} \gphpar^{\oc-s-1} \frac{1}{\sqrt{s}} \\
&= \tnode\setdia\nodepar\pbrack{2\setdia \dvgradf + \frac{1}{\mingrad}}\sum_{\oc=\grtint{\frac{\threshold}{2}}}^{\threshold} \sum_{s=1}^{\oc-1} \gphpar^{\oc-s-1} \frac{1}{\sqrt{\oc}} \frac{1}{\sqrt{s}} \\
&\overset{(k)}{\leq} \tnode\setdia\nodepar\pbrack{2\setdia \dvgradf + \frac{1}{\mingrad}}\sum_{\oc=\grtint{\frac{\threshold}{2}}}^{\threshold} \sum_{s=1}^{\oc-1} \frac{\gphpar^{\oc-s-1}}{s}\\
&\overset{(l)}{\leq}  \tnode\setdia\nodepar\pbrack{2\setdia \dvgradf + \frac{1}{\mingrad}}\sum_{s=1}^{\threshold-1}\frac{1}{s}\sum_{\oc=\max(\grtint{\frac{\threshold}{2}},s+1)}^{\threshold}  \gphpar^{\oc-s-1} \\
&\overset{(m)}{\leq}  \tnode\setdia\nodepar\pbrack{2\setdia \dvgradf + \frac{1}{\mingrad}}\sum_{s=1}^{\threshold-1}\frac{1}{s} \frac{\gphpar^{\max(\grtint{\frac{\threshold}{2}}-s-1,0)}}{1-\gphpar} \\
&=  \frac{\tnode\setdia\nodepar}{1-\gphpar}\pbrack{2\setdia \dvgradf + \frac{1}{\mingrad}}\pbrack{\sum_{s=1}^{\grtint{\frac{\threshold}{2}}-1}\frac{\gphpar^{\grtint{\frac{\threshold}{2}}-s-1}}{s}+\sum_{\grtint{\frac{\threshold}{2}}}^{\threshold-1}\frac{1}{s}}\\
&\overset{(n)}{\leq}  \frac{\tnode\setdia\nodepar}{1-\gphpar}\pbrack{2\setdia \dvgradf + \frac{1}{\mingrad}}\pbrack{\frac{1}{1-\gphpar}+2\ln{2}}, 
\end{align*}
where, (j) follows by substituting $\lrate{\oc}=\frac{\setdia}{\sqrt{\oc}}$, $\lrate{s}=\frac{\setdia}{\sqrt{s}}$ and $\errseq{s} = \frac{1}{\sqrt{s}}$, (k) follows because $\frac{1}{\sqrt{\oc}}\frac{1}{\sqrt{s}}\leq\frac{1}{s}$ for $\oc\geq s$, (l) follows via exchanging the summations, (m) follows by sum of geometric progression and the fact that $\gphpar\leq 1 $ and (n) follows from the inequalities $\sum_{s=1}^{\grtint{\frac{\threshold}{2}}-1}\frac{\gphpar^{\grtint{\frac{\threshold}{2}}-s-1}}{s}\leq\sum_{s=1}^{\grtint{\frac{\threshold}{2}}-1}\gphpar^{\grtint{\frac{\threshold}{2}}-s-1}\leq\frac{1}{1-\gphpar}$ and $\sum_{\grtint{\frac{\threshold}{2}}}^{\threshold-1}\frac{1}{s}\leq 2\ln{2}$. Similarly
\begin{align*}
F &=  \nodepar\sum_{\oc\in\fset{}} \lrate{\oc}\gphpar^{\oc-1} \sum_{\jnode\in\node} \norm{\dv{1}{\jnode}} \overset{(o)}{\leq} \nodepar\frac{\setdia}{1-\gphpar} \sum_{\jnode\in\node} \norm{\dv{1}{\jnode}},
\end{align*}
where, (o) is because $\sum_{\oc\in\fset{}} \lrate{\oc}\gphpar^{\oc-1}\leq\setdia\sum_{\oc\in\fset{}} \gphpar^{\oc-1}\leq\frac{\setdia}{1-\gphpar}$, where the last inequality is due to the infinite sum of a geometric progression. And,
\begin{align*}
G &= 2\sum_{\oc\in\fset{}}\lrate{\oc}\pbrack{2\lrate{\oc} \dvgradf + \frac{\errseq{\oc}}{\mingrad}} \\
&= 2\setdia\pbrack{2\setdia \dvgradf + \frac{1}{\mingrad}}\sum_{\oc\in\fset{}}\frac{1}{k} \\
&\leq 4\setdia\pbrack{2\setdia \dvgradf + \frac{1}{\mingrad}} \ln{2},
\end{align*}
where, the bound follow via substitution and upper bounding the sum as done previously. Using these we establish,
\begin{align*}
C&\leq \frac{\tnode\setdia\nodepar}{1-\gphpar}\pbrack{2\setdia \dvgradf + \frac{1}{\mingrad}}\pbrack{\frac{1}{1-\gphpar}+2\ln{2}} \\
&\hspace{20mm}+ \nodepar\frac{\setdia}{1-\gphpar} \sum_{\jnode\in\node} \norm{\dv{1}{\jnode}}+4\setdia\pbrack{2\setdia \dvgradf + \frac{1}{\mingrad}} \ln{2}
\end{align*}
Also, note that $D = G/2$, therefore
\begin{align*}
D \leq 2\pbrack{2\setdia^2\dvgradf+\frac{\setdia}{\mingrad}} \ln{2}
\end{align*}
By using the bounds for the terms A, B, C and D we get
\begin{align*}
 \objfunc{}{\bdv{\threshold}{\inode}} - \objfunc{}{\dv{*}{}} &\leq \frac{\tnode}{\sum_{\oc\in\fset{}}\lrate{\oc}}\pbrack{\frac{ \setdia}{ 2} + \maxgrad^2 \pbrack{\setdia^2 + \frac{\pbrack{\errseq{\oc}+\lrate{\oc}\dvgradf\mingrad}^2}{\mingrad^4}} \ln{2}}\\
&\hspace{5mm}+ \frac{2\dvgradf\tnode}{\sum_{\oc\in\fset{}}\lrate{\oc}}\frac{\tnode\setdia\nodepar}{1-\gphpar}\pbrack{2\setdia \dvgradf + \frac{1}{\mingrad}}\pbrack{\frac{1}{1-\gphpar}+2\ln{2}} \\
&\hspace{5mm}+ \frac{2\dvgradf\tnode}{\sum_{\oc\in\fset{}}\lrate{\oc}} \pbrack{\nodepar\frac{\setdia}{1-\gphpar} \sum_{\jnode\in\node} \norm{\dv{1}{\jnode}}+4\setdia\pbrack{2\setdia \dvgradf + \frac{1}{\mingrad}} \ln{2}}\\
&\hspace{5mm}+\frac{2\dvgradf\tnode}{\sum_{\oc\in\fset{}}\lrate{\oc}} \pbrack{2\setdia^2\dvgradf+\frac{\setdia}{\mingrad}} \ln{2}
\end{align*}
Now by substituting with the expression for $\lrate{\oc}$ and  using the lower bound $\sum_{\oc\in\fset{}}\frac{1}{\sqrt{k}}\geq (2-\sqrt{2})\sqrt{\threshold}$ in the above inequality and collecting all the constant term together as $\errop\pbrack{\gphpar,\nodepar,\dvgradf,\mingrad,\dvgradg,\maxgrad,\setdia,\tnode,\sum_{\jnode\in\node}\norm{\dv{1}{\jnode}}}$, we get the required bound.
\item[\namedlabel{thmproof3}{(\romannumeral 3)}] %
For the constraint function we have,
\begin{align*}
\max_{\uv{}{}\in\uvset{}} \confunc{}{\bdv{\threshold}{\inode},\uv{}{}} &= \max_{\uv{}{}\in\uvset{}} \confunc{}{\frac{1}{\sum_{\oc\in\fset{}}\lrate{\oc}}\sum_{\oc\in\fset{}}\lrate{\oc}\dv{\oc+1}{\inode},\uv{}{}}\\ 
&\overset{(p)}{\leq} \max_{\uv{}{}\in\uvset{}} \frac{1} {\sum_{\oc\in\fset{}}\lrate{\oc}}\sum_{\oc\in\fset{}}\lrate{\oc}\confunc{}{\dv{\oc+1}{\inode},\uv{}{}} \\
&= \frac{1} {\sum_{\oc\in\fset{}}\lrate{\oc}}\sum_{\oc\in\fset{}}\lrate{\oc} \max_{\uv{}{}\in\uvset{}} \confunc{}{\dv{\oc+1}{\inode},\uv{}{}} \\
&\overset{(q)}{\leq} \frac{1} {\sum_{\oc\in\fset{}}\lrate{\oc}}\sum_{\oc\in\fset{}}\lrate{\oc} \errseq{\oc+1}
\end{align*}
where (p) is due to the convexity of $\confunc{}$ in the first variable and (q) follows from the algorithm design which ensures that $\sup_{\uv{}{}\in\uvset{}} \confunc{}{\dv{\oc+1}{\inode},\uv{}{}}\leq \errseq{\oc+1}$. Now by substituting $\lrate{\oc} = \frac{\setdia}{\sqrt{\oc}}$ and $\errseq{\oc}=\frac{1}{\sqrt{\oc}}$, we have
\begin{align*}
\sup_{\uv{}{}\in\uvset{}} \confunc{}{\bdv{\threshold}{\inode},\uv{}{}} &\leq \frac{1} {\sum_{\oc\in\fset{}}\lrate{\oc}}\sum_{\oc\in\fset{}}\frac{\setdia}{\sqrt{\oc}} \frac{1}{\sqrt{\oc+1}}\leq \frac{\setdia} {\sum_{\oc\in\fset{}}\lrate{\oc}}\sum_{\oc\in\fset{}}\frac{1}{\oc}, 
\end{align*}
now, by using the inequalities: $\sum_{\oc\in\fset{}}\frac{1}{\sqrt{\oc}} \geq (2-\sqrt{2})\sqrt{\threshold}$ and $\sum_{\oc\in\fset{}}\frac{1}{\oc}\leq 2\ln{2}$, we get the required bound.
\end{enumerate} 
\end{proof}

The above results establish both optimality and feasibility guarantees for Algorithm~\ref{algo:dis_com}, and confirm that the algorithm terminates (i.e., the inner descent terminates finitely). Specifically, the suboptimality is of the order $\bigo{\frac{1}{\sqrt{\threshold}}}$ and so vanishes asymptotically with $K$. Similarly, the constraint violation also decreases with the number of outer iterations, with the worst-case infeasibility being of the same order, $\bigo{\frac{1}{\sqrt{\threshold}}}$. Note that, we have not established convergence to an optimizer in Theorem~\ref{thm:optimality}.

\section{Numerical experiments}
\label{sec:simulations}

In this section, we present numerical experiments to showcase the performance of Algorithm~\ref{algo:dis_com}. In Section~\ref{sec:num:optimization}, we consider a generic distributed optimization problem to illustrate the convergence properties of the algorithm. In Section~\ref{sec:num:control}, we consider a meta control control problem to demonstrate the application of the algorithm in a control setting.

\subsection{Example for the distributed semi-infinite optimization setup}
\label{sec:num:optimization}
Consider the following distributed optimization problem: 
\begin{equation}
\label{eq:sip_numerical}
\begin{aligned}
&\min_{\dv{}{0}, \dv{}{1}} &&\hspace{-8mm}\sum_{\inode=1}^{10} 0.1\pbrack{\dv{}{0}-a_{\inode}}^2+0.1\pbrack{\dv{}{1}-b_{\inode}}^2+\abs{\dv{}{0}+\dv{}{1}-4} -c_{\inode} \\
&\sbjto &&\hspace{-2mm}\begin{cases}
\dv{}{0},\dv{}{1}\in [-5,5], \\
d \dv{2}{0} + e \dv{}{1} -4 \leq 0, \\ \text{for all } d\in[0.5, 2.5] \text{ and } e\in[1, 3],
\end{cases}
\end{aligned}
\end{equation}
where, 
\begin{align*}
\{a_i\}_{\inode\in\until{10}} &= \setdef{-2, 3, -3, -5, -1, 0, 4, 2, -4, 1}, \\
\{b_i\}_{\inode\in\until{10}} &= \setdef{2, -2, 3, 5, 1, 0, -1, -3, 4, -4}, \\
\{c_i\}_{\inode\in\until{10}} &= \setdef{7, 3, 5, 1, 9, 11, 10, 14, 2.5, 12.5}.
\end{align*}
The objective function is distributed across ten different nodes on a network, and the optimizer at each node must satisfy the common semi-infinite constraint. Note that the objective function of the optimization problem~\eqref{eq:sip_numerical} is non-differentiable. We consider two network configurations: a directed line graph and a directed cycle. We use the method put forth in~\cite{SD-AA-AC-DC:22} to obtain the exact solution of this convex SIP (centralized); the optimal value is -33.3732 and the true optimizer is $x_0=0.53905$ and $x_1=1.09119$. It can be observed from Fig.\ \ref{fig:cycle_optimal} and Fig.\ \ref{fig:line_optimal} that, for both network configurations, the value of the objective function evaluated at the estimate that each node holds at a given iteration converges towards the actual optimal value as the number of iterations increases for both network configurations. This value after 20000 iterations across the nodes is
\begin{equation*}
\begin{aligned}
&(-33.3731,-33.3729,-33.3933,-33.3933,-33.3881,\\
&-33.3879, -33.3919, -33.3731,-33.3731, -33.3733),  
\end{aligned}
\end{equation*}
for nodes in the cycle configuration and
\begin{equation*}
\begin{aligned}
&(-33.3787, -33.3833, -33.3707, -33.3722, -33.3867, \\
&-33.3879, -33.3916, -33.3710, -33.3699, -33.3693),    
\end{aligned}
\end{equation*}
for nodes in the line configuration. The optimizer here is unique due to the strict convexity of the objective, and we plot the distance of the estimate at each node to this optimizer in Fig.\ \ref{fig:cycle_optimizer} and Fig.\ \ref{fig:line_optimizer}. Moreover, observe from Fig.\ \ref{fig:cycle_fes} and Fig.\ \ref{fig:line_fes} that the estimates at all nodes follow the feasibility tolerance sequence throughout the execution.
\begin{figure*}
\centering   
\begin{subfigure}[b]{0.4\textwidth}
\centering
\includegraphics[width=\linewidth]{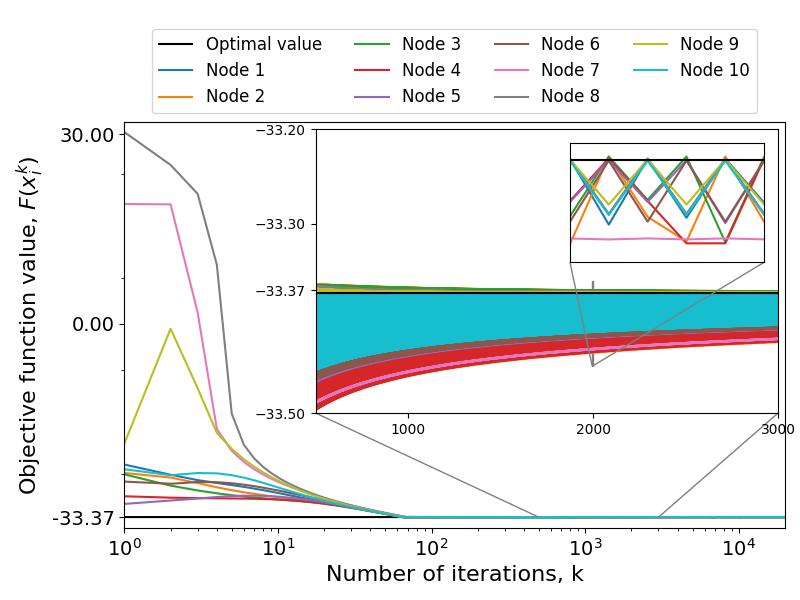}
\caption{{\footnotesize Objective function value - cycle configuration}}
\label{fig:cycle_optimal}
\end{subfigure}
\begin{subfigure}[b]{0.4\textwidth}
\centering
\includegraphics[width=\linewidth]{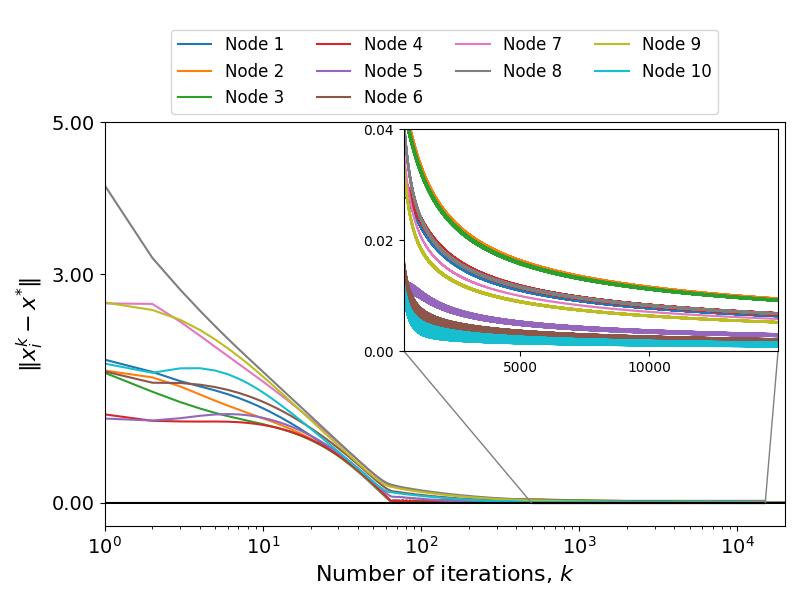}
\caption{{\footnotesize Error for the optimizer - cycle configuration}}
\label{fig:cycle_optimizer}
\end{subfigure}
\vskip\baselineskip
\begin{subfigure}[b]{0.4\textwidth}
\centering
\includegraphics[width=\linewidth]{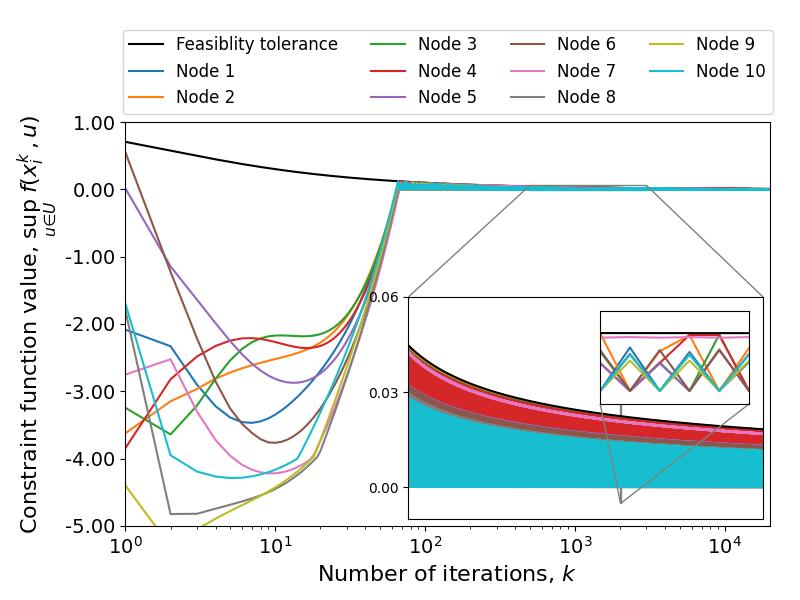}
\caption{{\footnotesize Feasibility of the estimate - cycle configuration}}
\label{fig:cycle_fes}
\end{subfigure}
\begin{subfigure}[b]{0.4\textwidth}
\centering
\includegraphics[width=\linewidth]{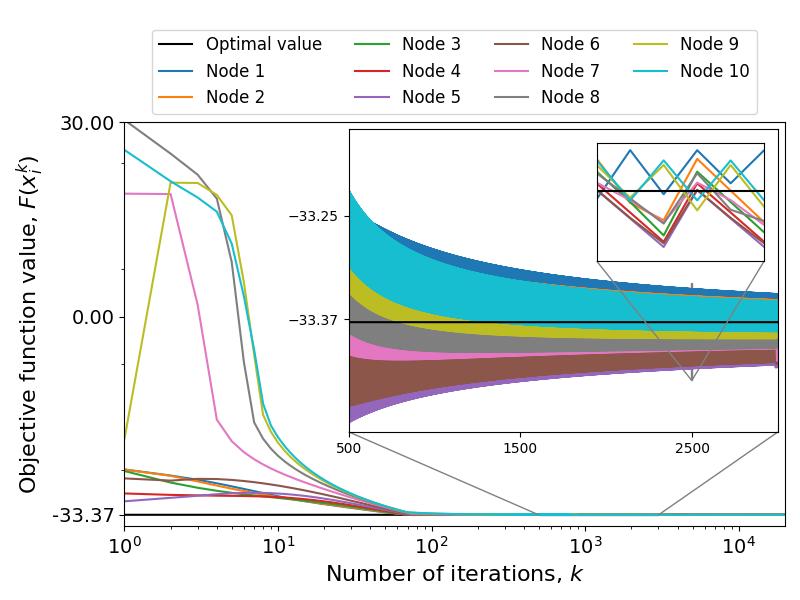}
\caption{{\footnotesize Objective function value - line configuration}}
\label{fig:line_optimal}
\end{subfigure}
\vskip\baselineskip
\begin{subfigure}[b]{0.4\textwidth}
\centering
\includegraphics[width=\linewidth]{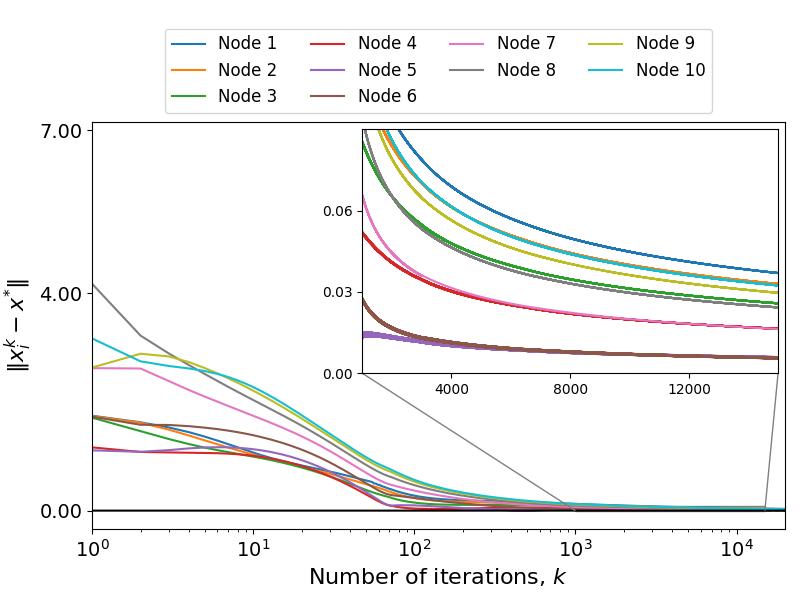}
\caption{{\footnotesize Error for the optimizer - line configuration}}
\label{fig:line_optimizer}
\end{subfigure}
\begin{subfigure}[b]{0.4\textwidth}
\centering
\includegraphics[width=\linewidth]{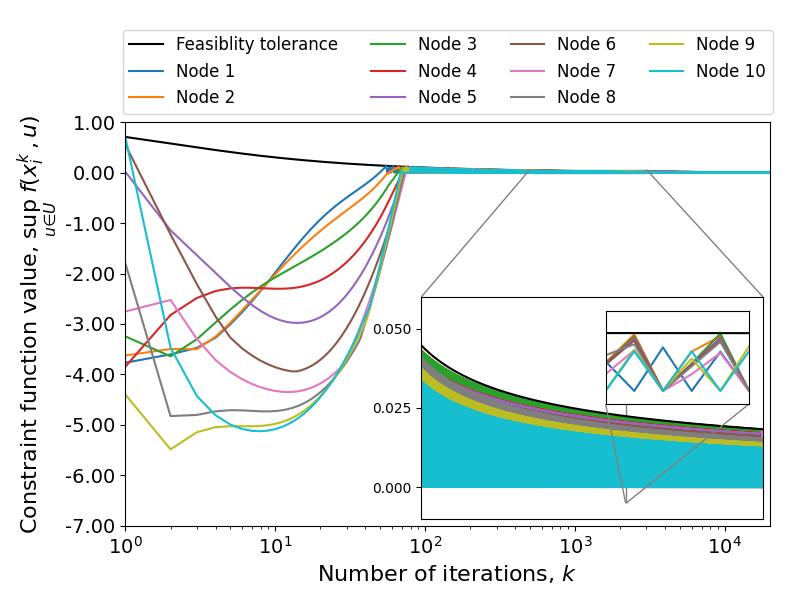}
\caption{{\footnotesize Feasibility of the estimate - line configuration}}
\label{fig:line_fes}
\end{subfigure}
\caption{Plots illustrating the application of Algorithm~\ref{algo:dis_com} on problem~\eqref{eq:sip_numerical}. 
 Panels (a), (b), and (c)  correspond to the cycle configuration of the network, and (d), (e), and (f) are for the line configuration. The evolution of the objective value at the estimates held by each node are displayed in (a) and (d). Similarly, the evolution of the error between the estimates and the optimizer is given in (b) and (e). Finally, the constraint violations at the estimates are shown in (c) and (f). These plots indicate that the algorithm reaches an optimizer asymptotically.}
\end{figure*}
\subsection{Comparison results}\label{sec:comparison}

We compared Algorithm~\ref{algo:dis_com} against two other  distributed methods that can handle semi-infinite constraints: the Distributed Cutting-Surface ADMM (DCSA) method~\cite{AC-AZ-GB-ARH:22}, and the Distributed Scenario-based Algorithm (DSA)\cite{KM-AF-SG-MP:18}. The comparison is conducted in the context of problem~\eqref{eq:sip_numerical}, where the network topology is the cycle digraph. In DCSA, each agent constructs progressively tighter outer approximations of the semi-infinite constraint by incrementally adding more constraints (termed cuts) sampled from the set $\mathbb{U}$, as the iterations progress. These approximations are then used in combination with a distributed ADMM method to solve the optimization problem~\eqref{eq:sip_central}. Increasing the number of cuts improves the feasibility and optimality guarantee of the obtained solution. On the other hand, DSC relies on a scenario-based approximation, where the constraints are sampled from $\mathbb{U}$ before the iterations start and remain fixed throughout the execution of the distribution algorithm. The distributed algorithm can be any method designed to solve  a deterministic distributed optimization problem. Naturally, the feasibility guarantee in this case depends on the number of initial samples and is therefore probabilistic in nature. The comparison provided here with our method focuses on numerical performance in terms of feasibility, optimality, and computational efficiency.

Fig.\ \ref{fig:comparison} shows the optimality and feasibility comparison. We plot the objective and constraint function values for the average of estimates $\widehat{x}^k$ defined in~\eqref{eq:avg-estimate} for Algorithm~\ref{algo:dis_com}, DCSA, and three instances of DSA with 50, 500, and 5000 scenarios distributed evenly among the ten nodes. While the DSA iterates initially converge faster toward optimality, DCSA and Algorithm~\ref{algo:dis_com} begin to show  comparable performance at around the $100^{\text{th}}$ iteration. As the iterations progress, both our method and DCSA outperform DSA and continue to approach the optimal value. This is consistent with the fact that both methods provably reach the optimal value asymptotically. The feasibility plot further supports this fact. Since the constraint samples in DSA are fixed from the start, constraint violations (if any) do not vanish as iterations progress; they only improve with the number of initially drawn samples. In contrast, both Algorithm~\ref{algo:dis_com} and DCSA reduce the constraint violation as iterations progress, achieving comparable levels of feasibility.

Fig.\ \ref{fig:time:comparison} compares the per-iteration computation time across all methods. Rather than plotting the raw per-iteration times, which exhibit significant variability, we report the cumulative average per-iteration time to obtain smoother and more interpretable plots. Let $\overline{T}_i^k$ be the time taken by node $i$ during the $\oc^{\text{th}}$ iteration of an approach. Then the per-iteration time is the average time taken by all nodes at the $\oc^{\text{th}}$ iteration, that is, $\overline{T}_{\mathrm{avg}}^k := V^{-1} \sum_{i \in [V]}  \overline{T}_i^k$.  The cumulative average per-iteration time at $\oc^{\text{th}}$ iteration  is computed as $\overline{T}^{\oc} := k^{-1} \sum_{s=1}^k \overline{T}_{\mathrm{avg}}^k$.

This running average provides a stable representation of computational effort per iteration across the distributed system. As shown in Fig.\ \ref{fig:time:comparison}, the per-iteration computation time for DCSA increases over the course of execution. This is due to the cumulative growth in the number of cutting planes, which leads to progressively larger subproblems. In contrast, the DSA methods exhibit nearly constant per-iteration times, as their computational complexity is determined by the fixed number of scenarios specified in advance. Algorithm~\ref{algo:dis_com} maintains a nearly constant and significantly lower per-iteration computation time throughout. This efficiency arises from the first-order nature of the method and the finite termination of its inner loop. Summarizing the above outlined results, we conclude that Algorithm~\ref{algo:dis_com} combines both features: asymptotic optimality as well as low per-iteration computation cost. 

\begin{figure*}
\centering   
\begin{subfigure}[b]{\textwidth}
\centering
\includegraphics[scale=0.42]{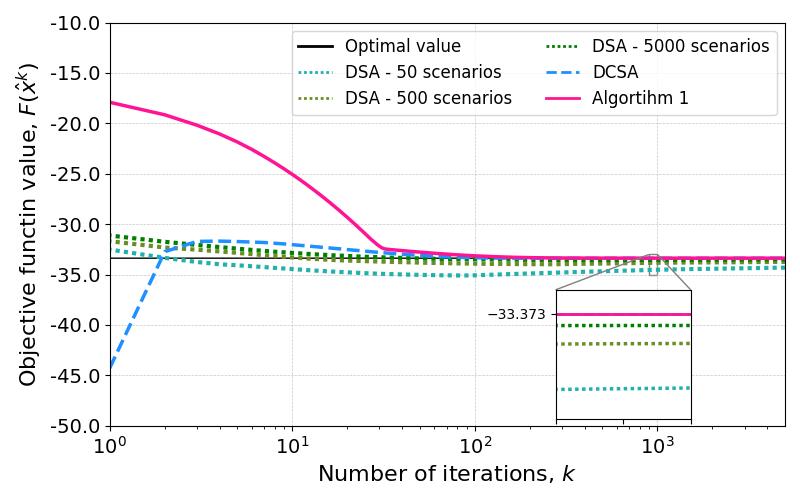}
\caption{{\footnotesize Objective function value - cycle configuration}}
\label{fig:dv:comparison}
\end{subfigure}
~\vskip\baselineskip
\begin{subfigure}[b]{\textwidth}
\centering
\includegraphics[scale=0.42]{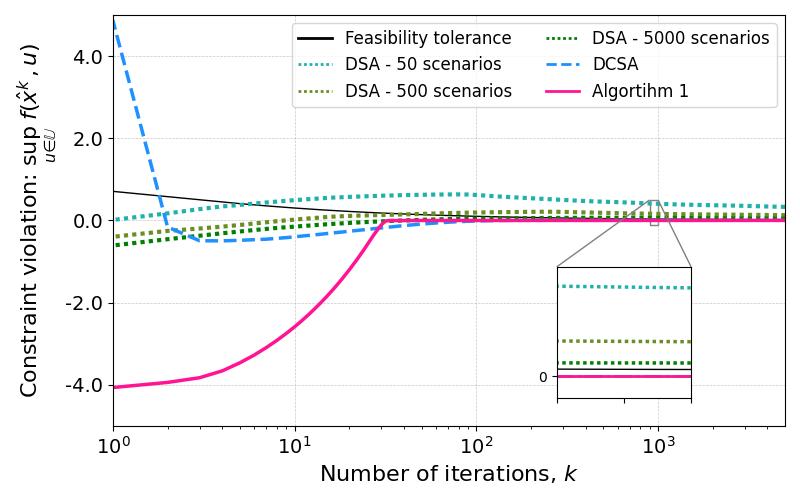}
\caption{{\footnotesize Feasibility of the estimates}}
\label{fig:fes:comparison} 
\end{subfigure}
\caption{Comparison of convergence behaviour of Algorithm~\ref{algo:dis_com}, Distributed Cutting-Surface ADMM~\cite{AC-AZ-GB-ARH:22} and the Distributed Scenario-based Algorithm (DSA) in~\cite{KM-AF-SG-MP:18}. Performance is evaluated based on both feasibility and optimality of the averaged estimate $\widehat{x}^k$ (i.e., the mean of local estimates across all nodes). Algorithm~\ref{algo:dis_com} and DCSA exhibit similar and strong performance, with trajectories that closely approach the optimal value and satisfy the feasibility tolerance. In contrast, the DSA methods with 50 and 500 scenarios show significant sub-optimality and infeasibility, highlighting the limitations of scenario-based methods when using too few samples. The DSA variant with 5000 scenarios improves in both aspects, yet still under performs slightly compared to Algorithm~\ref{algo:dis_com} and DCSA.}
\label{fig:comparison}
\end{figure*}
\begin{figure}
\centering
\includegraphics[width=0.7\linewidth]{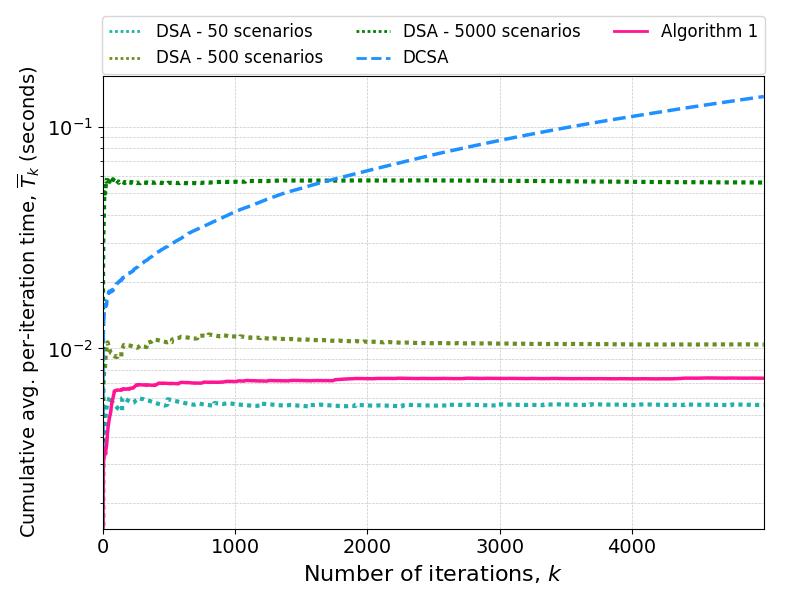}
\caption{Comparison of computational behaviour of Algorithm~\ref{algo:dis_com}, Distributed Cutting-Surface ADMM~\cite{AC-AZ-GB-ARH:22} and the Distributed Scenario-based Algorithm (DSA) in~\cite{KM-AF-SG-MP:18}. We plot the average time per iteration for each method, see Section~\ref{sec:comparison} for details. Algorithm~\ref{algo:dis_com} completes the full run (5000 iterations) in approximately 37 seconds per node on average, while the DCSA takes around 683 seconds per node and DSA with 5000 scenarios takes about 280 seconds per node. This highlights the significant computational advantage of Algorithm~\ref{algo:dis_com} in achieving comparable performance. Although DSA with 500 and 50 scenarios exhibits computation times similar to Algorithm~\ref{algo:dis_com}, the corresponding solution quality is noticeably lower (see Fig.\ \ref{fig:comparison}).}
\label{fig:time:comparison}
\end{figure}

\subsection{Robust meta control design}
\label{sec:num:control}
Consider the following parameter dependent discrete-time linear system:
\begin{equation}
\label{eq:sim:discon:dynamics}
    \sysdv^{t+1} = A(\sysnoi)\sysdv^{t} + B \sysin^{t}, \; t \in \{0\} \cup [99],
\end{equation}
where, for a time-step $t$, $\sysdv^t:=(\sysdv_{1}^t \; \sysdv_{2}^t)\in\R{2}{}$ is the state, $\sysin^t\in\R{}{}$ is the control input, and $\sysnoi\in\mathbb{W}:=[10,20]$ is an uncertain parameter, drawn i.i.d under uniform distribution on the set $\mathbb{W}$. The system matrices are of the form:
\begin{equation*}
A(\sysnoi)=\begin{bmatrix}
    1 & 0.01 \\ -0.01\sysnoi & 0.99
\end{bmatrix}, \quad B = \begin{bmatrix}
 0.0 \\ 0.01   
\end{bmatrix}.
\end{equation*}
We consider a scenario where four realizations of the above system are present at four different nodes of a network in cycle configuration. These realization are:
\begin{align*}
A_1&=\begin{bmatrix}
    1 & 0.01 \\ -0.1 & 0.99
\end{bmatrix},  &&\hspace{-25mm}A_2 = \begin{bmatrix}
    1 & 0.01 \\ -0.12 & 0.99
\end{bmatrix},
\\
A_3&=\begin{bmatrix}
    1 & 0.01 \\ -0.15 & 0.99
\end{bmatrix},  &&\hspace{-25mm}A_4=\begin{bmatrix}
    1 & 0.01 \\ -0.2 & 0.99
\end{bmatrix}.
\end{align*} 
Our objective is to determine a control sequence $\{\sysin^t\}_{t=0}^{99}$ that (a) minimizes the cumulative cost associated with these four realizations and (b) satisfies a terminal constraint when this control sequence is applied to any possible realization of the parameter dependent system~\eqref{eq:sim:discon:dynamics}. This formulation is motivated from meta learning frameworks~\cite{LFT-DZ-JA-HW:24,AA-MTT-CAU:24} for linear systems where the objective is to find control inputs that can be easily adapted to unseen realizations of an uncertain system. This objective is encoded in the control design as a terminal state constraint that needs to be satisfied for all possible realizations of the parametric uncertainty. This meta control design may be used to safely initialize a system with unknown dynamics and collect data before estimating the same.

Formally, the agents must agree on a common control sequence  $\{\sysin^t\}_{t=0}^{99}$ that minimizes the sum of their individual cost functions, given a prespecified common initial state $\sysdv^0 = \bar{\sysdv}$, while ensuring that the following semi-infinite constraint is satisfied by the terminal state:
\begin{align}\label{eq:sim:discon:con}
\norm{ A(\sysnoi)^{100}\sysdv^0+\sum_{t=0}^{99} A(\sysnoi)^{99-t}  B \sysin_{\inode}^{t}}^2 \le 1.5,  \text{for all }\sysnoi\in\mathbb{W}. 
\end{align}
The above description is captured in the following robust meta control problem:
\begin{align}
\label{eq:sim:discon:problem}
&\min_{(\sysin^t)_{t=0}^{99}} &&\sum_{\inode\in\until{4}}\left(\langle \sysdv_{\inode}^{100}, P_{\inode} \sysdv_{\inode}^{100} \rangle + \sum_{t=0}^{99} \langle \sysdv_{\inode}^{t}, Q_{\inode} \sysdv_{\inode}^{t} \rangle + \langle {\sysin^t}, R_{\inode} \sysin^{t} \rangle \right)\nonumber\\
& \, \, \text{s.t.} &&\begin{cases}
\sysdv_i^{t+1} = A_i\sysdv_i^{t} + B \sysin^{t}, \; t \in \{0\} \cup [99],
\\
\sysdv_{\inode}^{0} = (0.5\; 0), \;\text{for all }i\in\until{4},
\\
\text{the constraint~\eqref{eq:sim:discon:con}},
\end{cases} 
\end{align}
where the cost matrices are given by:
\begin{align*}
 (Q_{1},P_{1},R_{1}) &= \left(\begin{bmatrix}
  1.0 & 0.0 \\ 0.0 & 2.0   
 \end{bmatrix}, \begin{bmatrix}
1.0 & 0.0 \\ 0.0 & 2.0
 \end{bmatrix}, \begin{bmatrix}
0.1
 \end{bmatrix}\right), \\  
 (Q_{2},P_{2},R_{2}) &= \left(\begin{bmatrix}
  2.0 & 0.0 \\ 0.0 & 1.5   
 \end{bmatrix}, \begin{bmatrix}
2.0 & 0.0 \\ 0.0 & 1.5
 \end{bmatrix}, \begin{bmatrix}
1.0
 \end{bmatrix}\right), \\  
 (Q_{3},P_{3},R_{3}) &= \left(\begin{bmatrix}
  1.5 & 0.0 \\ 0.0 & 2.0   
 \end{bmatrix}, \begin{bmatrix}
1.5 & 0.0 \\ 0.0 & 2.0
 \end{bmatrix}, \begin{bmatrix}
1.0
 \end{bmatrix}\right), \\  
 (Q_{4},P_{4},R_{4}) &= \left(\begin{bmatrix}
  1.0 & 0.0 \\ 0.0 & 1.0   
 \end{bmatrix}, \begin{bmatrix}
1.0 & 0.0 \\ 0.0 & 1.0
 \end{bmatrix}, \begin{bmatrix}
0.1
 \end{bmatrix}\right).
\end{align*}
We used Algorithm~\ref{algo:dis_com} to solve the robust meta control problem~\eqref{eq:sim:discon:problem}. The meta control sequence obtained is depicted is depicted in Fig.\ \ref{fig:discon:cntrl}. Using this control sequence we plot in Fig.\ \ref{fig:discon:traj} 1000 state trajectories corresponding to the dynamics~\eqref{eq:sim:discon:dynamics}for 1000 different realizations of the uncertain parameter $\sysnoi$ (generated by uniformly randomly sampling on the set $\mathbb{W}$). As illustrated in Fig.\ \ref{fig:discon:control}, the robust meta control sequence successfully ensures constraint satisfaction across all realizations, demonstrating its effectiveness even on previously unseen instances of the system.

\begin{figure*}[ht]
\centering
\begin{subfigure}[b]{0.8\textwidth}
\includegraphics[width=\linewidth]{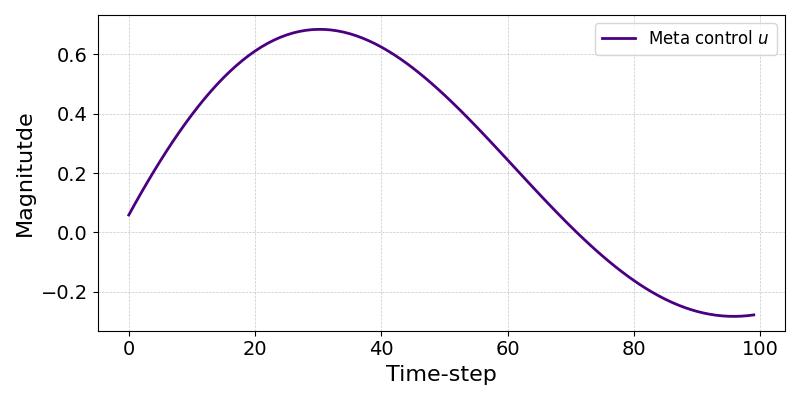} 
\caption{}
\label{fig:discon:cntrl}
\end{subfigure}
\vskip\baselineskip
\begin{subfigure}[b]{0.8\textwidth}
\includegraphics[width=\linewidth]{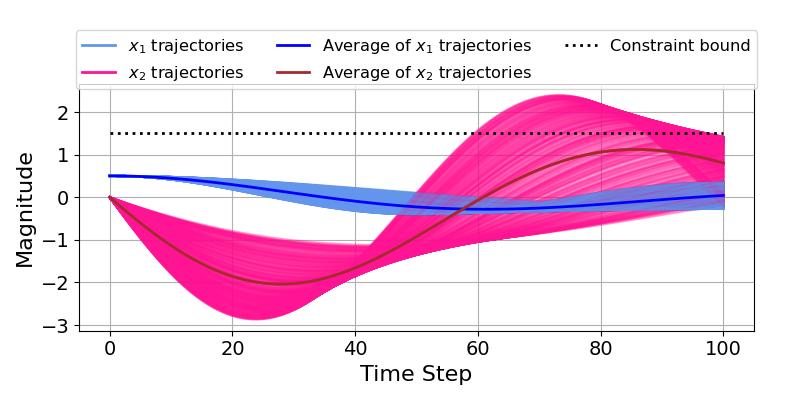} 
\caption{}
\label{fig:discon:traj}
\end{subfigure}
\caption{In Fig.\ \ref{fig:discon:cntrl}, we plot the robust meta control sequence $(\sysin^t)_{t=0}^{99}$ obtained by solving~\eqref{eq:sim:discon:problem} using Algorithm~\ref{algo:dis_com}. Fig.\ \ref{fig:discon:traj} shows 1000 state trajectories generated by applying this control sequence to the parameter-dependent system~\eqref{eq:sim:discon:dynamics}, with each trajectory corresponding to a random realization of the uncertain parameter. Notably, the control sequence satisfies the terminal state constraint~\eqref{eq:sim:discon:con} for all realizations, highlighting its robustness.}
\label{fig:discon:control}
\end{figure*}

\section{Conclusion}
\label{sec:conclusion}

We established a distributed algorithm for solving a convex semi-infinite program over a time-varying network. Algorithm~\ref{algo:dis_com} developed here is a first-order distributed method based on CoMirror descent for constrained optimization. We established its convergence guarantee, showing that the network achieves consensus and the optimal value and constraint violation for the estimates held by each node decay asymptotically to zero. We also established the  rate of this decay. The numerical experiments corroborate the algorithm's performance guarantees. In future work, we plan to extend our approach to distributed problems where the constraint functions are heterogeneous across nodes and they change with time. Moreover, we aim to refine the convergence rate for special class of objective functions.

\bibliographystyle{IEEEtran}
\bibliography{refs.bib}
\end{document}